\documentclass[12pt]{amsart}
\usepackage{amsmath,amssymb,latexsym,dsfont}
\usepackage[left=2.6cm,right=2.6cm,top=2.7cm,bottom=2.7cm]{geometry}

\usepackage{amsmath,amsfonts,amssymb,epsfig, textcomp}
\usepackage{graphicx}
\usepackage{hyperref, cleverref}
\usepackage{cases}
\usepackage{color}

\newtheorem{theorem}{Theorem}[section]
\newtheorem{lemma}{Lemma}[section]

\newtheorem{proposition}{Proposition}[section]
\newtheorem{corollary}{Corollary}[section]
\newtheorem{remark}{Remark}[section]

\numberwithin{equation}{section}

      \newcommand{\hu}{\hat u}
      \newcommand{\hv}{\hat v}
      
      \newcommand{\hmu}{\hat \mu}
      \newcommand{\heps}{\hat \eps}
      \newcommand{\hJ}{\hat J}
   \newcommand{\hE}{\hat E}

\newcommand{\cT}{\mathcal T}

      \newcommand{\thJ}{\hat{\mathrm{J}}}
      \newcommand{\tJ}{\mathrm{J}}

   \newcommand{\hH}{\hat H}

   \newcommand{\ha}{\hat a}

 \newcommand{\cF}{{\mathcal F}}

      \newcommand{\curl}{\operatorname{curl}}
      \newcommand{\dive}{\operatorname{div}}

      \newcommand{\eps}{\varepsilon}
      \newcommand{\mR}{\mathbb{R}}

      \newcommand{\mC}{\mathbb{C}}

      \newcommand{\tE}{\mathrm{E}}  
      \newcommand{\tH}{\mathrm{H}}
      \newcommand{\thE}{\hat{\mathrm{E}}}  
      \newcommand{\thH}{\hat{\mathrm{H}}}

 \newcommand{\hvarphi}{\hat \varphi}

      \newcommand{\supp}{\operatorname{supp}}

\newcommand{\bH}{\mathbf H}

      \makeatletter
      \def\@setcopyright{}
      \def\serieslogo@{}
      \makeatother

\title[Discreteness of Transmission Eigenvalues]{On the Discreteness of  Transmission Eigenvalues for  the Maxwell Equations}
\author{Fioralba Cakoni}
\address[Fioralba Cakoni]{Department of Mathematics, Rutgers University, 
	\newline\indent New Brunswick, NJ 08903, USA.}
\email{fc292@math.rutgers.edu}

\author{Hoai-Minh Nguyen}
\address[Hoai-Minh Nguyen]{Ecole Polytechnique F\'ed\'erale de Lausanne, EPFL,  SB, CAMA, 
\newline\indent Station 8,  CH-1015 Lausanne, Switzerland.}
\email{hoai-minh.nguyen@epfl.ch}

\begin{document}

\maketitle

\begin{abstract}
In this paper, we establish the discreteness of transmission eigenvalues for Maxwell's equations. More precisely, we show that the spectrum of the  transmission eigenvalue problem is discrete, if the electromagnetic parameters $\eps, \, \mu, \, \heps, \,  \hmu$ in the equations characterizing the inhomogeneity and  background, are  smooth in some neighborhood of  the boundary,  isotropic on the boundary,   and satisfy the conditions $\eps \neq \heps$,  $\mu \neq \hmu$, and $\eps/ \mu \neq \heps/  \hmu$  on the boundary. These are quite general assumptions on the coefficients which are easy to check. To our knowledge, our paper is the first to establish discreteness  of transmission eigenvalues for Maxwell's equations without assuming any restrictions on the sign combination of the contrasts  $\eps-\heps$ and  $\mu - \hmu$ near the boundary,  and allowing for all the electromagnetic parameters to be inhomogeneous and anisotropic, except for on the boundary where they are  isotropic but not necessarily constant as it is often assumed in the literature. 
\end{abstract}



\section{Introduction}

The  transmission eigenvalue problem is at the heart of  inverse scattering theory for inhomogeneous media.  This eigenvalue  problem is a late arrival  in scattering theory with its first appearance in \cite{old}, \cite{k22},  in connection with injectivity of the relative scattering operator. Transmission eigenvalues are  related to interrogating frequencies for which there is an incident field that doesn't scatterer by the medium.  The  transmission eigenvalue problem has a  deceptively simple formulation,  namely two elliptic PDEs in a bounded domain (one governs the wave propagation in the scattering medium  and the other in the background that occupies the support of the medium) that share the same Cauchy data on the boundary, but  presents a  perplexing mathematical structure. In particular, it is a non-selfadjoint eigenvalue problem for a non-strongly elliptic operator, hence the  investigation of its spectral properties becomes challenging.  Roughly, the spectral properties depend on the assumptions on the contrasts in the media (i.e. the difference of the respective coefficients in each of the equations) near the boundary.  Questions central to the inverse scattering theory include:  {\it discreteness} of the spectrum that is closely related to the determination of the support of  inhomogeneity from scattering data using linear sampling and factorization methods \cite{CCH16},  {\it location} of transmission eigenvalues in the complex plane that is essential to the development of the time domain linear sampling method \cite{CMV20}, and the {\it existence} of transmission eigenvalues as well as the {\it accurate determination} of real  transmission eigenvalues from scattering data, which has became important since real transmission eigenvalues could be used to obtain information about the material properties of the scattering media.  We refer the reader to \cite{CCH16} for a recent and self-contained introduction to the topic.

 This paper concerns the discreteness  and location of transmission eigenvalues in the scattering of  time-harmonic electromagnetic waves by an  inhomogeneous (possibly anisotropic) medium of bounded support.  Let us  introduce the  mathematical formulation of  the  electromagnetic transmission eigenvalue problem. To this end,  let $\Omega$ be an open,  bounded  subset of $\mR^3$ representing the support of the inhomogeneity,  which we assume to be of class $C^2$. Let $\eps,\,  \mu, \, \heps, \,  \hmu$ be ($3 \times 3$) symmetric, uniformly elliptic,   matrix-valued  functions  defined in $\Omega$ with $L^\infty(\Omega)$ entries.  A complex number $\omega$ is called an eigenvalue of the  transmission eigenvalue problem, or a {\it transmission eigenvalue}, associated with $\eps, \,  \mu, \,\heps, \, \hmu$  in $\Omega$ if there exists a non-zero solution $(E, H, \hE, \hH) \in [L^2(\Omega)]^{12}$ of the following system  
\begin{equation}\label{sys-ITE}
\left\{\begin{array}{c}
\nabla \times E =  i \omega \mu H  \\[6pt]
\nabla \times H = - i \omega \eps E
\end{array}\right.  \, \mbox{ in } \Omega,\qquad  \quad \left\{\begin{array}{c}
\nabla \times \hE = i \omega  \hmu \hH \\[6pt]
\nabla \times \hH = - i \omega  \heps \hE 
\end{array}\right.  \, \mbox{ in } \Omega,
\end{equation}
\begin{equation}\label{bdry-ITE}
(\hE - E) \times \nu = 0\,  \mbox{ on } \partial \Omega, \quad \mbox{ and } \quad (\hH - H) \times \nu = 0 \, \mbox{ on } \partial \Omega,   
\end{equation}
where $\nu$ denotes the outward unit normal vector to $\partial \Omega$.

\medskip 

The main result  that we prove in this paper is stated in \Cref{thm-main} below. For the reader convenience we  first must  clarify some terminology used in the formulation of  this theorem. A $3 \times 3$ matrix-valued function $M$ defined in a  subset $O \subset \mR^3$ is called isotropic at $x \in O$ if it is proportional  to the identity matrix at $x$, i.e., $M(x) = m I$ for some scalar $m = m(x)$ where $I$ denotes the $3 \times 3$ identity matrix. In this case, for the notational ease, we also denote $m(x)$ by $M(x)$. If $M$ is isotropic for $x \in O$, then $M$ is said to be isotropic in $O$. Condition~\eqref{thm-main-cond} below is understood under the convention $m(x) = M(x)$.

\begin{theorem}\label{thm-main} Assume that 
\begin{enumerate}
\item[i)] $\eps, \,  \mu, \, \heps, \,  \hmu$  are of class $C^1$ in some neighborhood of $\partial \Omega$, 
\item[ii)]  $\eps, \,  \mu, \, \heps, \,  \hmu$ are isotropic on $\partial \Omega$, 
\item[iii)]
\begin{equation}\label{thm-main-cond}
\eps \neq \heps, \quad \mu \neq \hmu, \quad \eps/ \mu \neq \heps/  \hmu \quad \mbox{ on } \partial \Omega. 
\end{equation}
\end{enumerate}
The set of the  transmission  eigenvalues  of  \eqref{sys-ITE} and \eqref{bdry-ITE} is discrete with $\infty$ as the only  possible accumulation point. 
\end{theorem}

The analysis used in the proof of \Cref{thm-main} also allows us to obtain the following result on the  transmission eigenvalue free region of the complex plane ${\mathbb C}$.

\begin{proposition}\label{pro} Assume that $\eps, \,  \mu, \, \heps, \,  \hmu$  are of class $C^1$ in some neighborhood of $\partial \Omega$,  isotropic  on $\partial \Omega$,  and 
\begin{equation}\label{thm-main-cond2}
\eps \neq \heps, \quad \mu \neq \hmu, \quad \eps/ \mu \neq \heps/  \hmu \quad \mbox{ on } \partial \Omega. 
\end{equation}
For $\gamma >0$, there exists $\omega_0 > 0$ such that if $\omega \in \mC$ with $|\Im(\omega^2)| \ge \gamma |\omega|^2$ and $|\omega| \ge \omega_0$, then $\omega$ is not a transmission eigenvalue. 
\end{proposition}

Here and and in what follows, for $z \in \mC$, let  $\Im(z)$ denote the imaginary part of $z$.

\begin{remark}
\em{Since $\gamma>0$  can be chosen arbitrary small, the result of Proposition \ref{pro} together the fact that $\infty$ is  the only accumulation point of the transmission eigenvalues proven in \Cref{thm-main}, imply that all the transmission eigenvalues $\omega$, but finitely many, lie in a wedge of arbitrary small angle.}
\end{remark}

The structure of spectrum of the transmission eigenvalue problem is better understood in the case of scalar inhomogeneous Helmoltz equations. In this case, the  transmission eigenvalue problem can be stated as follows. Let $d \ge 2$ and $\Omega$ be an open,  bounded  Lipschitz subset of $\mR^d$ . Let $A_1, A_2$ be two ($d \times d$) symmetric, uniformly elliptic,   matrix-valued bounded function defined in $\Omega$ and $\Sigma_1$ and $\Sigma_2$ be two bounded positive functions defined in $\Omega$.  A complex number $\omega$ is called an eigenvalue of the  transmission eigenvalue problem, or a transmission eigenvalue, if there exists a non-zero solution $(u_1, u_2)$ of the system 
 \begin{equation}\label{pro1a}  
 \left\{\begin{array}{lll}
 \dive(A_1 \nabla u_1) + \omega^2 \Sigma_1 u_1= 0 ~~&\text{ in}~\Omega, \\[6pt]
 \dive(A_2 \nabla u_2) + \omega^2 \Sigma_2 u_2= 0 ~~&\text{ in}~\Omega, 
 \end{array} \right. 
 \end{equation}
 \begin{equation}\label{pro1b}
  u_1 =u_2, \quad  A_1 \nabla u_1\cdot \nu = A_2 \nabla u_2\cdot \nu ~\text{ on } \partial \Omega.  
 \end{equation}
The discreteness of transmission eigenvalues for  the Helmholtz equation has been  investigated extensively in the literature. The first  discreteness result appeared in \cite{first}, whereas  \cite{MinhHung17} proves the state-of-the-art results  on the discreteness of transmission eigenvalues for anisotropic background and inhomogeneity under most general assumptions on the coefficients using Fourier and multiplier approaches. More specifically, it is shown in \cite{MinhHung17} that the transmission eigenvalue problem has a discrete spectrum if the coefficients are smooth  only near the boundary, and 
\begin{enumerate}
\item[i)] $A_1(x),  \, A_2(x)$ satisfy the complementing boundary condition with respect to  $\nu(x)$ for all $x \in \partial \Omega$, i.e., 
for all $x \in \partial \Omega$ and for all $\xi \in \mR^d \setminus \{0\}$ with  $\xi \cdot \nu = 0$, we have 
\begin{equation*}
\langle A_2 \nu, \nu \rangle \langle A_2 \xi, \xi \rangle  - \langle A_2 \nu,  \xi \rangle^2 \neq  \langle A_1 \nu, \nu \rangle \langle A_1 \xi, \xi \rangle  - \langle A_1 \nu, \xi \rangle^2, 
\end{equation*}
\item[ii)]  $\big\langle  A_1(x) \nu(x), \nu(x) \big\rangle \Sigma_1(x) \neq \big\langle  A_2(x) \nu(x), \nu(x) \big\rangle \Sigma_2(x)$ for all $x \in \partial \Omega$. 
\end{enumerate}
Additional results in \cite{MinhHung17}, also include various combinations of the sign of contrasts $A_1-A_2$ and $\Sigma_1-\Sigma_2$ on the boundary. Previous results on discreteness can be found in  \cite{BCH11,  LV12, Sylvester12} and references therein. We must emphasize that the conditions i) and ii) are more general than 
simply one sign contrasts $A_2 - A_1$ and/or $\Sigma_2 - \Sigma_1$ near the boundary. To complete the picture on the  transmission eigenvalue problem in the scalar case, we remark that the first answer to the existence of transmission eigenvalues for one sign contrast in $\Omega$ was given in \cite{PS06}  where the authors showed the existence of a few real transmission eigenvalues for the index of refraction sufficiently large,  followed by \cite{exist}, \cite{kirsch}, which prove the existence of infinite real transmission eigenvalues removing the size restriction on the index. Completeness of  transmission eigenfunctions and first estimates on the  counting function are shown in \cite{Robiano13}, \cite{Robbiano16}  for $C^\infty$  boundary and coefficients since they use semiclassical analysis and pseudo-differential calculus. Again in $C^\infty$ setting, \cite{Vodev15}, \cite{Vodev18} prove the sharpest known results in the scalar case on eigenvalue free zones and Weyl's law for the scalar  case  improving an earlier result by \cite{lassi}. 

The story of the  transmission eigenvalue problem for Maxwell's equations is not as complete as  for the scalar case discussed above. Of the first results on discreteness is given by Haddar in \cite{Haddar04}  where it considers the case of $\mu = \heps = \hmu = I $, and $\eps - I $ invertible in $\Omega$.  Chesnel in \cite{Chesnel12} employs the so-called $T$-coercivity  to prove  discreteness when $\heps = \hmu = I $,  and $\eps - I $ and $\mu^{-1} - I $ are both greater than $cI$ or both less $-c I$ in $\Omega$ for some positive constant $c$ in a neighborhood of $\partial \Omega$. Cakoni et al. in  \cite{CHM15} use  an  integral equation approach to study discreteness for  the case when $\mu = \heps = \hmu = I $ and the matrix valued $\eps$  becomes constant not equal to $1$ near the boundary.  \Cref{thm-main}  therefore adds to this list  quite general conditions on the coefficients for which the discreteness holds. To our knowledge, our paper is the first to establish discreteness  of transmission eigenvalues for Maxwell's equations without assuming any restrictions on the sign combination of the contrasts  $\eps-\heps$ and  $\mu - \hmu$ near the boundary (as long as they do not change sign up the boundary),  and allowing for all the electromagnetic parameters to be inhomogeneous and anisotropic, except for on the boundary where they are  isotropic but not necessarily constant as it is often assumed in the literature.  For the case of electromagnetic  transmission eigenvalue problems, other type of results are rather limited, and we refer the reader to  \cite{exist}  for the existence of real transmission eigenvalues and \cite{HM18} for the completeness of eigenfunctions for the setting related to the one in \cite{CHM15} mentioned above. 

The analysis in this paper is inspired  by the concept of complementary conditions suggested by  Agmon, Douglis, and Nirenberg in their celebrated papers \cite{ADNI, ADNII} for elliptic systems. For  Maxwell's equations,  the complementary condition for the Cauchy problems has been recently investigated in \cite{NgSil} for general anisotropic coefficients in the context of negative index metamaterials. To be able to apply the theory of complementing conditions to the Maxwell equations, various forms of the Poincar\'e lemma and Helmholtz decomposition are used with a suitable implementation of local charts.  The analysis in this paper  is in the spirit of the one developed in \cite{MinhHung17}. The idea is to show that the following system 
\begin{equation}\label{sys-ITE-E}
\left\{\begin{array}{c}
\nabla \times E =  i \omega \mu H  + J_e\mbox{ in } \Omega, \\[6pt]
\nabla \times H = - i \omega \eps E + J_m  \mbox{ in } \Omega, 
\end{array}\right. \quad \left\{\begin{array}{c}
\nabla \times \hE = i \omega  \hmu \hH + \hJ_e  \mbox{ in } \Omega, \\[6pt]
\nabla \times \hH = - i \omega  \heps \hE + \hJ_m \mbox{ in } \Omega, 
\end{array}\right.
\end{equation}
\begin{equation}\label{bdry-ITE-E}
(\hE - E) \times \nu = 0 \mbox{ on } \partial \Omega, \quad \mbox{ and } \quad (\hH - H) \times \nu = 0 \mbox{ on } \partial \Omega, 
\end{equation}
is well-posed for some $\omega \in \mC$ where $(J_e, J_m, \hJ_e, \hJ_m)$ is the input, which belongs to  an appropriate functional space. Moreover, a key fact is to prove that the corresponding transformation which maps the input  $(J_e, J_m, \hJ_e, \hJ_m)$ to the output $(E, H, \hE, \hH)$ is   compact. It is worth mentioning that the compactness is one of the crucial/critical difference between the study of the Maxwell equations and the Helmholtz equation. In our analysis, the functional space for the input is well-chosen so that the compactness property holds (see \eqref{def-Hspace}) for $\omega$ in some domain. For example, these facts hold under the assumptions of \Cref{thm-main}  provided that  $i \omega = |\omega| e^{i \pi/4}$, i.e., $\omega =  |\omega| e^{- i \pi/4}$ and $|\omega|$ is large. To this end, we analyze the corresponding Cauchy problem with constant coefficients in a  half-space  (\Cref{pro-HS}).  Using the decay of Maxwell equations (\Cref{lem-decay}), we can prove the uniqueness for  \eqref{sys-ITE-E} and \eqref{bdry-ITE-E}. To establish the existence of a solution, the limiting absorption principle is used,  and several processes involving the Fredholm theory of compact operator are applied. Again the choice of the functional space to ensure the compactness plays an important role here in order to use the Fredholm theory. Deriving \eqref{thm-main-cond} and handling the compactness are the key difference in the analysis of this paper and the one for the scalar case \cite{MinhHung17}. 

The Cauchy problem also naturally appears in the context of negative index materials after using reflections as initiated in \cite{Ng-Complementary}.  The well-posedness and the limiting absorption principle for the Helmholtz equations with sign-changing coefficients
was developed in \cite{Ng-WP} using the Fourier and multiplier approach. Recently, with Sil, the second author investigate these problems for the Maxwell equations \cite{NgSil}. Both papers \cite{Ng-WP},  \cite{NgSil}  deal with the stability question of negative index materials and are the starting point for the analysis of the discreteness of transmission eigenvalues for the Helmholtz equation \cite{MinhHung17} and Maxwell's equations in this work.   Other aspects and applications of negative index materials involving the stability and instability the Cauchy problem \eqref{sys-ITE-E} and \eqref{bdry-ITE-E} are discussed in \cite{Ng-Superlensing-Maxwell, Ng-CALR-O-M, Ng-Negative-Cloaking-M} and the references therein. 

The paper is organized as follows. In \Cref{sect-notation}, we introduce several notations used frequently in this paper. \Cref{sect-HS} is devoted to the analysis in the half space. The main result in this section is \Cref{pro-HS}. Condition~\eqref{thm-main-cond} will appear very  naturally there. Finally, we present the proof of \Cref{thm-main} in \Cref{sect-proof}. 
The choice of the right functional space plays an important role there so that the Fredholm theory can be applied.  The proof of \Cref{pro} is also given in this section.

\section{Notations}\label{sect-notation}

The following notations are used frequently throughout the paper. Denote 
$$
\mR^3_+ = \Big\{x = (x_1, x_2, x_3) \in \mR^3; \; x_3 > 0 \Big\}
$$
and 
$$
\mR^3_0 = \Big\{x = (x_1, x_2, x_3) \in \mR^3; \; x_3 = 0 \Big\}. 
$$
Let $\Omega$ be a bounded,  open subset of $\mR^3$ and of class $C^2$,  or $\Omega = \mR^3_+$. We define the spaces
\begin{equation*}
H(\curl, \Omega) = \Big\{u \in [L^2(\Omega)]^3; \nabla \times u \in [L^2(\Omega)]^3 \Big\}, 
\end{equation*}
\begin{equation*}
H_0(\curl, \Omega) = \Big\{u \in H(\curl, \Omega) ;  u \times \nu = 0 \mbox{ on } \partial \Omega \Big\}, 
\end{equation*}
\begin{equation*}
H(\dive, \Omega) = \Big\{u \in [L^2(\Omega)]^3 ;  \dive u  \in L^2(\Omega) \Big\}. 
\end{equation*}
Set $\Gamma = \partial \Omega$,  and for $s = -1/2$, or $1/2$, define the trace space
\begin{equation*}
H^{s}_{\dive} (\Gamma) = \Big\{u \in [H^{s}(\Gamma)]^3;  u \cdot \nu = 0 \mbox{ and } \dive_{\Gamma} u  \in H^{s}(\partial \Omega) \Big\}.  
\end{equation*}
For a vector field $u$ defined in a subset of $\mR^3$, $u_j$ denotes its $j$-th component for $1 \le j \le 3$. 
We also denote, for $s > 0$,  
\begin{equation}\label{def-Os}
\Omega_s = \Big\{x \in \Omega; \mbox{dist}(x, \partial \Omega) < s \Big\}. 
\end{equation}

\section{Analysis on a half space} \label{sect-HS}
In order to simplify presentation, we let $k \in \mC$  be $k:= i \omega$. Let  $\eps$, $\mu$, $\heps$, $\hmu$ be four symmetric, uniformly elliptic matrix-valued  functions defined in $\mR^3_+$. In this section, 
we are interested in the following Cauchy problem for Maxwell's equations in $\mR^3_+$, with
$J_e, J_m, \hJ_e, \hJ_m \in L^2(\mR^3_+)$ and  $f_e, f_m \in H^{-1/2}_{\dive} (\mR^3_0)$, 
\begin{equation}\label{sys-C}
\left\{\begin{array}{c}
\nabla \times E = k \mu H + J_e \mbox{ in } \mR^3_+, \\[6pt]
\nabla \times H = - k \eps E + J_m \mbox{ in } \mR^3_+, 
\end{array}\right. \quad \left\{\begin{array}{c}
\nabla \times \hE = k \hmu \hH + \hJ_e \mbox{ in } \mR^3_+, \\[6pt]
\nabla \times \hH = - k \heps \hE + \hJ_m \mbox{ in } \mR^3_+, 
\end{array}\right.
\end{equation}
and \footnote{$e_3 = (0, 0, 1) \in \mR^3$.}
\begin{equation}\label{bdry-C}
(\hE - E) \times e_3 = f_e \mbox{ on } \mR^3_0, \quad \mbox{ and } \quad (\hH - H) \times e_3 = f_m \mbox{ on } \mR^3_0 .  
\end{equation}

We begin with proving the following lemma.
\begin{lemma}\label{lem-HS}  Let  $\gamma > 0$ and  $k  \in \mC$ with $\big| \Im (k^2) \big| \ge \gamma |k|^2 $ and $|k| \ge 1$. Furthermore, let $\Lambda \ge 1$ and $\eps$, $\mu$ be two positive constants such that $\Lambda^{-1} \le \eps, \,  \mu \le \Lambda$. For $J_e, J_m \in [L^2(\mR^3_+)]^3$,  there exists a unique solution $(E, H) \in [L^2(\mR^3)]^{6}$  of the system 
\begin{equation}\label{lem-HS-sys}
\left\{\begin{array}{c}
\nabla \times E= k \mu H  \mbox{ in } \mR^3_+, \\[6pt]
\nabla \times H = - k \eps E + J_m \mbox{ in } \mR^3_+, \\[6pt]
E \times e_3 = 0  \mbox{ on } \mR^3_0. 
\end{array}\right. 
\end{equation}
Moreover, for some positive constant $C$ depending only on $\Lambda$ and $\gamma$, 
\begin{equation}\label{lem-HS-est1}
 \| (E, H) \|_{L^2(\mR^3_+)} \le \frac{C}{|k|}   \| J_m \|_{L^2(\mR^3_+)}, 
\end{equation}
and  if $\dive J_m \in L^2(\Omega)$,  then 
\begin{equation}\label{lem-HS-est2}
\| (E, H) \|_{H^1(\mR^3_+)} \le C  \left(  \| J_m\|_{L^2(\mR^3_+)} + \frac{1}{|k|}  \| \dive J_m \|_{L^2(\mR^3_+)} \right).  
\end{equation}
\end{lemma}

\begin{remark} \rm We emphasize here that the constant  $C$ appearing  in \eqref{lem-HS-est1} and \eqref{lem-HS-est2} is independent of $k$.  
\end{remark}


\begin{proof} We have, from the system of $(E, H)$, 
\begin{equation}\label{lem-HS-sysE}
\nabla \times (\nabla \times E) + k^2 \eps \mu E =  k \mu J_m \mbox{ in } \mR^3_+. 
\end{equation}
Multiplying \eqref{lem-HS-sysE} by $\bar \varphi$ (the conjugate of $\varphi$) with $\varphi \in H_0(\curl, \mR^3_+)$,  and integrating by parts yields 
\begin{equation}\label{lem-HS-bilinear}
\int_{\mR^3_+} \langle \nabla \times E,  \nabla \times \varphi \rangle + k^2 \eps \mu \int_{\mR^3} \langle E,  \varphi \rangle  
= \int_{\mR^3} k \mu \langle J_m, \varphi \rangle. 
\end{equation}
Take $\varphi = E$. Since $\big| \Im (k^2) \big| \ge \gamma |k|^2 $ and $|k| \ge 1$, after considering the imaginary part and the real part of \eqref{lem-HS-bilinear}, we obtain 
$$
\int_{\mR^3_+} |\nabla \times E|^2 + |k|^2 |E|^2 \le  C  \int_{\mR^3_+} |J_m|^2, 
$$
which implies \eqref{lem-HS-est1} since $\nabla \times E = k \mu H$ in $\mR^3_+$. The uniqueness of $(E, H)$ follows.  

To derive \eqref{lem-HS-est2}, we note that 
$$
\| \nabla \times E\|_{L^2(\mR^3_+)} \le C \| J_m \|_{L^2(\mR^3_+)}, 
$$
$$
\| \dive E\|_{L^2(\mR^3_+)} \le \frac{C}{|k |} \| \dive J_m \|_{L^2(\mR^3_+)}, 
$$
$$
\| E\|_{L^2(\mR^3_+)} \le C \| J_m \|_{L^2(\mR^3_+)}, 
$$
$$
E \times e_3 = 0 \mbox{ on } \mR^3_+. 
$$
It follows from the Gaffney inequality, see e.g. \cite[Theorems 3.7]{GR86}, \cite[Theorem 1]{CDS18},  that  $E \in H^1(\mR^3_+)]^3$ and 
\begin{equation}\label{lem-HS-e1}
\| E \|_{H^1(\mR^3_+)} \le C  \| J_m \|_{L^2(\mR^3_+)} + \frac{C}{|k|} \| \dive J_m \|_{L^2(\mR^3_+)}. 
\end{equation}
We also have 
$$
\| \nabla \times H \|_{L^2(\mR^3_+)} \le C \| J_m \|_{L^2(\mR^3_+)}, 
$$
$$
\| \dive H\|_{L^2(\mR^3_+)}  = 0, 
$$
$$
\| H\|_{L^2(\mR^3_+)} \le C \| J_m \|_{L^2(\mR^3_+)}, 
$$
and, since $E \times e_3 = 0 $ on $\mR^3_+$,   
$$
H \cdot e_3 = 0 \mbox{ on } \mR^3_+. 
$$
It follows from Gaffney inequality again,  see e.g. \cite[Theorems 3.9]{GR86}, \cite[Theorem 1]{CDS18}, that $H \in H^1(\mR^3_+)$ and 
\begin{equation}\label{lem-HS-e2}
\| H \|_{H^1(\mR^3_+)} \le C  \| J_m  \|_{L^2(\mR^3_+)} + \frac{C}{|k|} \| \dive J_e \|_{L^2(\mR^3_+)}. 
\end{equation}
Combining \eqref{lem-HS-e1} and \eqref{lem-HS-e2} yields \eqref{lem-HS-est2}.

To prove the existence  of $(E, H)$, we first apply the Lax-Milgram theory for variational formula given in  \eqref{lem-HS-bilinear} where the bilinear form is defined by the LHS and the linear functional is defined by the RHS in the Hilbert space $H_0(\curl, \mR^3_+)$. We then derive that  there exists  a solution $E \in H_0(\curl, \mR^3_+)$ of \eqref{lem-HS-sysE}. Set 
$$
H = \frac{1}{k \mu} \nabla \times E \mbox{ in } \mR^3_+. 
$$
Then 
$$
\nabla \times E = k \mu H  \mbox{ in } \mR^3_+, 
$$
and 
$$
\nabla \times H =  \frac{1}{k \mu} \nabla \times  \big(\nabla \times E \big) = - k \eps E + J_m \mbox{ in } \mR^3_+. 
$$
In other words, $(E, H) \in [L^2(\mR^3_+)]^6$ is a solution of \eqref{lem-HS-sys}.  The proof is complete. 
\end{proof}

We now state the main result of this section, which plays a key role in the proof of \Cref{thm-main}. 

\begin{proposition}\label{pro-HS}  Let $\alpha \in \mC$ with $\alpha^2 \in \mR$ and $|\alpha| = 1$,  $\gamma > 0$,  $k  \in \mC$ with $|\Im(k^2)| \ge \gamma |k|^2$, $|\Im(\alpha^2 k^2)| \ge \gamma |k|^2$,  and $|k| \ge 1$, and let $\Lambda \ge 1$  and $\eps$, $\mu$, $\heps$, $\hmu$ be four positive constants such that 
$$
\Lambda^{-1} \le \eps, \,  \mu,  \, \heps, \,  \hmu \le \Lambda. 
$$
Assume that, for some $\Lambda_1 > 0$ 
$$
|\eps -  \heps| \ge \Lambda_1, \quad |\mu - \hmu| \ge \Lambda_1, \quad \mbox{ and } \quad |\eps/ \mu -  \heps/ \hmu| \ge \Lambda_1. 
$$
Let $J_e, J_m, \hJ_e, \hJ_m \in H(\dive, \mR^3_+)$ with $
(J_{e, 3} - \hJ_{e, 3}, J_{m, 3} - \hJ_{m, 3}) \in [H^{1/2}(\mR^3_0)]^2$, 
and let $f_e, f_m \in H^{1/2}(\dive,  \mR^3_0)$. 
There exists a unique solution $(E, H, \hE, \hH) \in [L^2(\mR^3)]^{12}$  of  the system 
\begin{equation}\label{sys-C*}
\left\{\begin{array}{c}
\nabla \times E = k \mu H + J_e \mbox{ in } \mR^3_+, \\[6pt]
\nabla \times H = - k \eps E + J_m \mbox{ in } \mR^3_+, 
\end{array}\right. \quad \left\{\begin{array}{c}
\nabla \times \hE =  \alpha k \hmu \hH + \hJ_e \mbox{ in } \mR^3_+, \\[6pt]
\nabla \times \hH = - \alpha k \heps \hE + \hJ_m \mbox{ in } \mR^3_+, 
\end{array}\right.
\end{equation}
\begin{equation}\label{bdry-C*}
(\hE - E) \times e_3 = f_e \mbox{ on } \mR^3_0, \quad \mbox{ and } \quad (\hH - H) \times e_3 = f_m \mbox{ on } \mR^3_0.  
\end{equation}
Moreover,  we have 
\begin{align}\label{pro-HS-est}
C \Big( \| (E, H, \hE, \hH)  & \|_{H^1(\mR^3_+)} + |k| \,  \| (E, H, \hE, \hH) \|_{L^2(\mR^3_+)} \Big) \\[6pt]
\le \quad  & \| (J_e, J_m, \hJ_e, \hJ_m)\|_{L^2(\mR^3_+)}  
+ \frac{1}{|k|}  \| (\dive J_e, \dive J_m, \dive \hJ_e, \dive \hJ_m)\|_{L^2(\mR^3_+)} \nonumber \\[6pt]
& + \frac{1}{|k|}  \| (J_{e, 3} - \hJ_{e, 3}, J_{m, 3} - \hJ_{m, 3}) \|_{H^{1/2}(\mR^3_0)} 
+  |k|^{1/2} \|(f_e , f_m) \|_{L^2(\mR^3_0)} \nonumber\\[6pt]  
& + \|(f_e , f_m) \|_{H^{1/2}(\mR^3_0)} + 
\frac{1}{|k|} \|(\dive_{\Gamma} f_e , \dive_{\Gamma} f_m) \|_{H^{1/2} (\mR^3_0)}, \nonumber
\end{align}
for some positive constant $C$ depending only on  $\gamma$, $\Lambda$,  and $\Lambda_1$. 
\end{proposition}
 
Recall that,  by our convention, $J_{e, 3}, \, J_{m, 3}, \, \hJ_{e, 3}, \hJ_{m, 3}$ denote the third component of $J_{e}, \, J_{m}, \, \hJ_{e}, \hJ_{m}$. It is worth noting that the constant $C$ is independent of $k$. 
 
\begin{proof} Let $(E^1, H^1), \, (E^2, H^2), \, (\hE^1, \hH^1), \, (\hE^2, \hH^2) \in [L^2(\mR^3_+)]^6$ be respectively the unique solutions  of the following systems
\begin{equation*}
\left\{\begin{array}{c}
\nabla \times E^1= k \mu H^1  \mbox{ in } \mR^3_+, \\[6pt]
\nabla \times H^1 = - k \eps E^1 + J_m \mbox{ in } \mR^3_+, \\[6pt]
E^1 \times e_3 = 0  \mbox{ on } \mR^3_0,  
\end{array}\right. \quad \left\{\begin{array}{c}
\nabla \times \hE^1 = \alpha k \hmu \hH^1  \mbox{ in } \mR^3_+, \\[6pt]
\nabla \times \hH^1 = - \alpha k \heps \hE^1 + \hJ_m \mbox{ in } \mR^3_+, \\[6pt]
\hE^1 \times e_3 = 0 \mbox{ on } \mR^3_0, 
\end{array}\right.
\end{equation*}
\begin{equation*}
\left\{\begin{array}{c}
\nabla \times E^2= k \mu H^2 + J_e \mbox{ in } \mR^3_+, \\[6pt]
\nabla \times H^2 = - k \eps E^1 \mbox{ in } \mR^3_+, \\[6pt]
H^2 \times e_3 = 0  \mbox{ on } \mR^3_0,  
\end{array}\right. \quad \left\{\begin{array}{c}
\nabla \times \hE^2 = \alpha k \hmu \hH^2 + \hJ_e \mbox{ in } \mR^3_+, \\[6pt]
\nabla \times \hH^2 = - \alpha k \heps \hE^2 \mbox{ in } \mR^3_+, \\[6pt]
\hH^2 \times e_3 = 0 \mbox{ on } \mR^3_0. 
\end{array}\right.
\end{equation*}
Applying \Cref{lem-HS}, we obtain 
\begin{multline}\label{pro-HS-e1}
\| (E^1, H^1, E^2, H^2, \hE^1, \hH^1, \hE^2, \hH^2) \|_{H^1(\mR^3_+)} \\[6pt]
 + |k| \, \| (E^1, H^1, E^2, H^2, \hE^1, \hH^1, \hE^2, \hH^2) \|_{L^2(\mR^3_+)} \\[6pt]
\le  C \Big(  \| (J_e, J_m, \hJ_e, \hJ_m)\|_{L^2(\mR^3_+)} + 
\frac{1}{|k|}   \| (\dive J_e, \dive J_m, \dive \hJ_e, \dive \hJ_m)\|_{L^2(\mR^3_+)} \Big). 
\end{multline}
From the trace theory, we derive from \eqref{pro-HS-e1} that 
\begin{multline}\label{pro-HS-e2}
\| (E^1, H^1, E^2, H^2, \hE^1, \hH^1, \hE^2, \hH^2) \|_{H^{1/2}(\mR^3_0)} \\[6pt] \le C \| (E^1, H^1, E^2, H^2, \hE^1, \hH^1, \hE^2, \hH^2) \|_{H^1(\mR^3_+)} \\[6pt]
\le  C \Big(  \| (J_e, J_m, \hJ_e, \hJ_m)\|_{L^2(\mR^3_+)} + 
\frac{1}{|k|}   \| (\dive J_e, \dive J_m, \dive \hJ_e, \dive \hJ_m)\|_{L^2(\mR^3_+)} \Big). 
\end{multline}

We have
\begin{align*}
\|\dive_{\Gamma} (H^1 \times  e_3) -& \dive_{\Gamma} (\hH^1 \times  e_3)\|_{H^{1/2}(\mR^3_0)} \\[6pt]
  = & \; \|(\nabla \times H^1)  \cdot e_3  - (\nabla \times \hH^1)  \cdot e_3 \|_{H^{1/2}(\mR^3_0)} \\[6pt]
 = &\;
\| (- k \eps E^1 + J_m)  \cdot e_3 - (- \alpha k \heps \hE^1 + \hJ_m)  \cdot e_3 \|_{H^{1/2}(\mR^3_0)} \\[6pt]
 \le & \;C |k| \,  \| (E^1, \hE^1) \|_{H^{1/2}(\mR^3_0)} + C \|J_{m, 3} - \hJ_{m, 3} \|_{H^{1/2}(\mR^3_0)}. 
\end{align*}
It follows from  \eqref{pro-HS-e2} that 
\begin{multline}\label{pro-HS-e3}
\|\dive_{\Gamma} (H^1 \times  e_3) - \dive_{\Gamma} (\hH^1 \times  e_3)\|_{H^{1/2}(\mR^3_0)} 
\le C |k| 
 \Big( \| (J_e, J_m, \hJ_e, \hJ_m)\|_{L^2(\mR^3_+)}   \\[6pt]
+ \frac{1}{|k|}   \| (\dive J_e, \dive J_m, \dive \hJ_e, \dive \hJ_m)\|_{L^2(\mR^3_+)} + \frac{1}{|k|} \|J_{m, 3} - \hJ_{m, 3} \|_{H^{1/2}(\mR^3_0)} \Big). 
\end{multline}

Similarly, we obtain 
\begin{multline}\label{pro-HS-e4}
\|\dive_{\Gamma} (E^1 \times  e_3) - \dive_{\Gamma} (\hE^1 \times  e_3)\|_{H^{1/2}(\mR^3_0)} \le 
C |k| 
 \Big( \| (J_e, J_m, \hJ_e, \hJ_m)\|_{L^2(\mR^3_+)}  \\[6pt] + 
\frac{1}{|k|}   \| (\dive J_e, \dive J_m, \dive \hJ_e, \dive \hJ_m)\|_{L^2(\mR^3_+)} + \frac{1}{|k|} \|J_{e, 3} - \hJ_{e, 3} \|_{H^{1/2}(\mR^3_0)} \Big). 
\end{multline}

Using the fact, for $u \in H^1(\mR^3_+)$,  
\begin{equation}\label{ineq-interpolation}
\int_{\mR^3_0} |u|^2 \le 2 \int_{\mR^3_+} |u| |\partial_{x_3} u| \le 2 \| u\|_{L^2(\mR^3_+)} \| \nabla u\|_{L^2(\mR^3_+)}, 
\end{equation}
we have 
\begin{multline*}
\|(E^1, H^1, E^2, H^2, \hE^1, \hH^1, \hE^2, \hH^2) \|_{L^2(\mR^3_0)} \\[6pt]\le C \| (E^1, H^1, E^2, H^2, \hE^1, \hH^1, \hE^2, \hH^2) \|_{L^2(\mR^3_+)}^{1/2} \| (E^1, H^1, E^2, H^2, \hE^1, \hH^1, \hE^2, \hH^2) \|_{H^1(\mR^3_+)}^{1/2}.  
\end{multline*}
This yields 
\begin{multline}\label{pro-HS-e5}
k^{1/2}\|(E^1, H^1, E^2, H^2, \hE^1, \hH^1, \hE^2, \hH^2) \|_{L^2(\mR^3_0)} \\[6pt]
\le C |k| \,  \| (E^1, H^1, \hE^1, \hH^1) \|_{L^2(\mR^3_+)} + C \| (E^1, H^1, \hE^1, \hH^1) \|_{H^1(\mR^3_+)}.  
\end{multline}

By considering $(E - E^1 - E^2 , H - H^1 - H^2, \hE - \hE^1 - \hE^2, \hH - \hH^1 - \hH^2)$, from \eqref{pro-HS-e1}, \eqref{pro-HS-e2}, \eqref{pro-HS-e3},  \eqref{pro-HS-e4}, and  \eqref{pro-HS-e5}, w.l.o.g. one might assume that 
$$
J_e = J_m = \hJ_e = \hJ_m = 0 \mbox{ in } \mR^3_+. 
$$
This will be assumed later on. Thus 
\begin{equation}\label{pro-HS-sys-C-*}
\left\{\begin{array}{c}
\nabla \times E = k \mu H \mbox{ in } \mR^3_+, \\[6pt]
\nabla \times H = - k \eps E  \mbox{ in } \mR^3_+, 
\end{array}\right. \quad \left\{\begin{array}{c}
\nabla \times \hE = \alpha k \hmu \hH  \mbox{ in } \mR^3_+, \\[6pt]
\nabla \times \hH = - \alpha k \heps \hE  \mbox{ in } \mR^3_+, 
\end{array}\right.
\end{equation}
\begin{equation}\label{pro-HS-bdry-C-*}
(\hE - E) \times e_3 = f_e \mbox{ on } \mR^3_0, \quad \mbox{ and } \quad (\hH - H) \times e_3 = f_m \mbox{ on } \mR^3_0.  
\end{equation}

Using the identity for a vector field $A$
$$
\nabla \times (\nabla \times A) = \nabla (\nabla \cdot A) - \Delta A, 
$$
we obtain  the following equations for $E$ and $\hE$
\begin{equation}\label{pro-HS-EE}
\left\{\begin{array}{c} \Delta E - k^2 \eps \mu E = 0 \mbox{ in } \mR^3_+, \\[6pt]
\Delta \hE - \alpha^2 k^2 \heps \hmu \hE = 0 \mbox{ in } \mR^3_+
\end{array}\right. 
\end{equation}
(recall that here  the coefficients are all constants). In the following, we denote the Fourier transform with respect to $(x_1, x_2) \in \mR^2$ of an appropriate function $u: \mR^3_+ \to \mC$ by $u^{\cF}$, i.e., 
$$
u^{\cF}(\xi, x_3) = \frac{1}{2 \pi} \int_{\mR^2} u(x_1, x_2, x_3) e^{ -  i (x_1 \xi_1 + x_2 \xi_2)} \, d x_1 \, d x_2 \mbox{ for } (\xi, x_3) =  (\xi_1, \xi_2, x_3) \in \mR^3_+. 
$$
Similar notation is used for an appropriate function defined on $\mR^3_0$. 

Consider the first two equations of the system for $E$ and the first two equations of the system for $\hE$ in \eqref{pro-HS-EE}. Solving these equations using the Fourier transform with respect to $(x_1, x_2)$ yields 
\begin{equation}\label{pro-HS-E}
E^{\cF}_j(\xi, x_3) = a_j(\xi) e^{- x_3 \sqrt{|\xi|^2 + k^2 \eps \mu }} \quad \mbox{ in } \mR^3_+, 
\end{equation}
\begin{equation}\label{pro-HS-hE}
\hE^{\cF}_j(\xi, x_3) = \ha_j(\xi) e^{- x_3 \sqrt{|\xi|^2 + \alpha^2 k^2 \heps \hmu }}  \quad \mbox{ in } \mR^3_+, 
\end{equation}
for $j=1, 2$,  where  
$$
a_j(\xi) = E^{\cF}_j(\xi, 0) \quad \mbox{ and } \quad \ha_j(\xi) = \hE^{\cF}_j(\xi, 0) \mbox{ for } \xi \in \mR^2. 
$$
We then have, with $a = (a_1, a_2)$ and $\ha = (\ha_1, \ha_2)$,  
\begin{equation}\label{pro-HS-ab1}
\ha(\xi) - a(\xi) = h (\xi) \mbox{ where } h(\xi) = - f_e^{\cF} (\xi, 0) \times e_3.  
\end{equation}
Here and in what follows, we identity a vector $(y_1, y_2, 0) \in \mR^3_0$ with  $(y_1, y_2) \in \mR^2$. 

Since $\dive E = 0$ in $\mR^3_+$, it follows that 
$$
\partial_{x_3} E_3 = - (\partial_{x_1} E_1 + \partial_{x_2} E_2) \quad  \mbox{ in } \mR^3_+.  
$$
This implies 
$$
\partial_{x_3} E_3^{\cF} (\xi, x_3) = -   i \xi_1 E_1^{\cF}   (\xi, x_3) - i \xi_2 E_2^{\cF}  (\xi, x_3) \quad \mbox{ in } \mR^3_+. 
$$
Using \eqref{pro-HS-E}, we obtain 
\begin{equation*}
\partial_{x_3} E_3^{\cF}(\xi, x_3) = -  i \xi \cdot a(\xi) e^{- x_3 \sqrt{|\xi|^2 + k^2 \eps \mu}} \quad \mbox{ in } \mR^3_+. 
\end{equation*}
We thus get
\begin{equation}\label{pro-HS-E3}
E_3^{\cF}(\xi, x_3) = - \int_{x_3}^\infty i \xi \cdot a(\xi) e^{- s \sqrt{|\xi|^2 + k^2 \eps \mu}} \, d s =   \frac{ i \xi \cdot a(\xi) e^{- x_3 \sqrt{|\xi|^2 + k^2 \eps \mu}}}{\sqrt{|\xi|^2 + k^2 \eps \mu}} \quad \mbox{ in } \mR^3_+. 
\end{equation}

Similarly, we have 
\begin{equation}\label{pro-HS-hE3}
\hE_3^{\cF}(\xi, x_3) =     \frac{ i  \xi \cdot \ha(\xi) e^{- x_3 \sqrt{|\xi|^2 + \alpha^2 k^2 \heps \hmu} }}{\sqrt{|\xi|^2 + \alpha^2 k^2 \heps \hmu}} \quad \mbox{ in } \mR^3_+. 
\end{equation}

Since $\hH \times e_3 - H \times e_3 = f_m$ on $\mR^3_0$, and $\nabla \times H = -  k \eps E$ and $\nabla \times \hH = - \alpha k \heps \hE$ in $\mR^3_+$,  it follows that   
$$
\alpha  \heps \hE_3 - \eps E_3  = - \frac{1}{k} \dive_{\mR^3_0} f_m \quad \mbox{ on } \mR^3_0. 
$$
Using \eqref{pro-HS-E3} and \eqref{pro-HS-hE3}, we derive that 
\begin{equation}\label{pro-HS-ab2}
 \frac{\alpha \heps  \xi \cdot \ha(\xi)}{\sqrt{|\xi|^2 + \alpha^2 k^2 \heps \hmu}} - \frac{\eps  \xi \cdot a(\xi)}{\sqrt{|\xi|^2 + k^2 \eps \mu}} = g: = \big( \frac{i}{k} \dive_{\mR^3_0} f_m  \big)^{\cF} \quad \mbox{ on } \mR^2. 
\end{equation}
Combining \eqref{pro-HS-ab1} and \eqref{pro-HS-ab2}, and noting $a = \ha - h$,  yield, on $\mR^2$,  
\begin{equation*}
\xi \cdot \ha \left( \frac{\alpha \heps}{\sqrt{|\xi|^2 + \alpha k^2 \heps \hmu}} -  \frac{ \eps }{\sqrt{|\xi|^2 +  k^2 \eps \mu}}  \right)  =  -  \frac{\eps}{\sqrt{|\xi|^2 + k^2 \eps \mu}} \xi \cdot h + g, 
\end{equation*}
which implies 
\begin{multline*}
\xi \cdot \ha  = \frac{ \sqrt{|\xi|^2 +k^2 \eps \mu } \sqrt{|\xi|^2 + \alpha^2 k^2 \heps \hmu } \Big(\eps \sqrt{|\xi|^2 + \alpha^2 k^2 \heps \hmu} + \alpha \heps \sqrt{|\xi|^2 + k^2 \eps \mu}  \Big)}{(\alpha^2 \heps^2 - \eps^2) |\xi|^2 + \alpha^2 k^2 \eps \heps \mu \hmu (\heps/ \hmu - \eps/ \mu)} \times \\[6pt]
\times \left( - \frac{\eps}{\sqrt{|\xi|^2 + k^2 \eps \mu}} \xi \cdot h + g \right). 
\end{multline*}
Since  $\eps \neq \heps$, $\eps/\mu \neq \heps/ \hmu$, $\alpha^2  = \pm 1$,  and $|\Im (k^2)| \ge  \gamma |k|^2$, $|k| \ge 1$, we get 
\begin{equation}\label{pro-HS-em}
|(\alpha^2 \heps^2 - \eps^2) |\xi|^2 + \alpha^2 k^2 \eps \heps \mu \hmu (\heps/ \hmu - \eps/ \mu)| \ge C (|\xi|^2 + |k|^2). 
\end{equation}
We deduce that  
\begin{equation*}
|\xi \cdot \ha(\xi) | \le C \Big(|\xi \cdot h(\xi)| + \sqrt{|\xi|^2 + |k|^2} |g(\xi)| \Big), 
\end{equation*}
which yields, since $a = \ha - h$,  
\begin{equation}\label{pro-HS-est1-1}
|\xi \cdot a(\xi) | + |\xi \cdot \ha(\xi) | \le C \Big(|\xi \cdot h(\xi)| + \sqrt{|\xi|^2 + |k|^2} |g(\xi)| \Big). 
\end{equation}

We have, in $\mR^3_+$, 
\begin{equation*}
k \mu H_1 = \partial_{x_2} E_3 - \partial_{x_3} E_2, \quad \alpha k \hmu \hH_1 = \partial_{x_2} \hE_3 - \partial_{x_3} \hE_2. 
\end{equation*}
Since $\hH_1 - H_1 = f_{m, 2} : = f_m \cdot e_2$  with $e_2 = (0, 1, 0)$ on $\mR^3_0$, it follows from \eqref{pro-HS-E}, \eqref{pro-HS-hE}, \eqref{pro-HS-E3}, \eqref{pro-HS-hE3} that 
\begin{multline*}
 \frac{1}{\alpha \hmu} \left(- \frac{ \xi_2   \xi \cdot \ha(\xi) }{\sqrt{|\xi|^2 +\alpha^2  k^2 \heps \hmu}} + \sqrt{|\xi|^2 + \alpha^2 k^2 \heps \hmu} \ha_2(\xi)   \right) \\[6pt]
 =\frac{1}{\mu} \left( - \frac{ \xi_2   \xi \cdot a(\xi) }{\sqrt{|\xi|^2 + k^2 \eps \mu}} + \sqrt{|\xi|^2 + k^2 \eps \mu} a_2(\xi)  \right) + k f_{m, 2}^{\cF} (\xi). 
\end{multline*}
We derive from \eqref{pro-HS-ab1} that 
\begin{multline}
\frac{1}{\alpha \hmu} \sqrt{|\xi|^2 + \alpha^2 k^2 \heps \hmu} \ha_2(\xi)  - \frac{1}{\mu} \sqrt{|\xi|^2 + k^2 \eps \mu} \ha_2(\xi) \\[6pt]
= \left(\frac{ \xi \cdot \ha(\xi) }{\alpha \hmu \sqrt{|\xi|^2 + \alpha^2 k^2 \heps \hmu}  } - \frac{ \xi \cdot a(\xi) }{\mu \sqrt{|\xi|^2 + k^2 \eps \mu} } \right) \xi_2 -  \frac{1}{\mu} \sqrt{|\xi|^2 + k^2 \eps \mu} h_2(\xi)
 + k f_{m, 2}^{\cF} (\xi). 
\end{multline}
We thus obtain
\begin{multline}\label{pro-HS-E*}
\ha_2(\xi) = \alpha \frac{ \mu \sqrt{|\xi|^2 + \alpha^2 k^2 \heps \hmu} +  \alpha \hmu \sqrt{|\xi|^2 + k^2 \eps \mu}}{ (\mu^2 - \alpha^2 \hmu^2) |\xi|^2 + \alpha^2 k^2 \eps \heps \mu \hmu (\mu/ \eps - \hmu/ \heps) } \\[6pt]
\times 
\left\{ \left(\frac{\mu \xi \cdot \ha (\xi)}{\alpha \sqrt{|\xi|^2 + \alpha^2 k^2 \heps \hmu}  } - \frac{\hmu \xi \cdot a(\xi)}{\sqrt{|\xi|^2 + k^2 \eps \mu} } \right)  \xi_2  - \hmu \sqrt{|\xi|^2 + k^2 \eps \mu} h_2(\xi) +  k \mu \hmu  f_{m, 2}^{\cF}(\xi) \right\}. 
\end{multline}
Since  $\mu \neq \hmu$, $\eps/\mu \neq \heps/ \hmu$, $\alpha^2 = \pm 1$, and $|\Im( k^2)| \ge \gamma |k|^2$, $|k| \ge 1$, we get 
\begin{equation}\label{pro-HS-em-2}
|(\mu^2 - \alpha^2 \hmu^2) |\xi|^2 + \alpha^2 k^2 \eps \heps \mu \hmu (\mu/ \eps - \hmu/ \heps)| \ge C (|\xi|^2 + |k|^2). 
\end{equation}
Using \eqref{pro-HS-em-2}, we derive from \eqref{pro-HS-E*} that 
$$
|\ha_2(\xi)| \le  \frac{C |\xi|}{|\xi|^2 + |k|^2} \Big( |\xi \cdot a (\xi)|  + |\xi \cdot \ha (\xi)|  \Big) +  C \Big( |h_2(\xi)| +  |f_{m, 2}^{\cF}(\xi)| \Big), 
$$
which yields, since $\ha - a = h$,
\begin{equation}\label{pro-HS-est1-2}
|a_2(\xi)| + |\ha_2(\xi)| \le  \frac{C |\xi|}{|\xi|^2 + |k|^2} \Big( |\xi \cdot a (\xi)|  + |\xi \cdot \ha (\xi)|  \Big) +  C \Big( |h(\xi)| +  |f_{m}^{\cF}(\xi)| \Big). 
\end{equation}

Combining  \eqref{pro-HS-est1-1} and \eqref{pro-HS-est1-2}
yields 
\begin{equation*}
 (|a_2(\xi)|^2 + |\ha_2(\xi)|^2 ) \sqrt{|\xi|^{2} + |k|^2} 
\le C  (|h|^2 + |f_{m}^{\cF}|^2\big) \sqrt{|\xi|^2 + |k|^2} + C |g(\xi)|^2 |\xi|. 
\end{equation*}
From the definition of $g$ and $h$, we obtain  
\begin{multline}\label{pro-HS-est2}
\int_{\mR^2} (|a_2(\xi)|^2 + |\ha_2(\xi)|^2 ) \sqrt{|\xi|^{2} + |k|^2} \\[6pt]
\le C \left( \| (f_e, f_m) \|_{H^{1/2}(\mR^3_0)}^2 + |k|  \, \| (f_e, f_m) \|_{L^2(\mR^3_0)}^2 + \frac{1}{|k|^2}  \| (\dive_\Gamma f_e, \dive_\Gamma f_m) \|_{H^{1/2}(\mR^3_0)}^2 \right). 
\end{multline}

Similarly, we reach  
\begin{multline}\label{pro-HS-est3}
\int_{\mR^2} (|a_1(\xi)|^2 + |\ha_1(\xi)|^2 ) \sqrt{|\xi|^{2} + |k|^2} \\[6pt]
\le C \left( \| (f_e, f_m) \|_{H^{1/2}(\mR^3_0)}^2 + |k|  \, \| (f_e, f_m) \|_{L^2(\mR^3_0)}^2 + \frac{1}{|k|^2}  \| (\dive_\Gamma f_e, \dive_\Gamma f_m) \|_{H^{1/2}(\mR^3_0)}^2 \right). 
\end{multline}

On the other hand, from \eqref{pro-HS-E}, \eqref{pro-HS-hE}, \eqref{pro-HS-E3}, and  \eqref{pro-HS-hE3},  we have 
\begin{equation}\label{pro-HS-est4}
\int_{\mR^3_+} |(\nabla E, \nabla \hE)|^2  + |k|^2 |(E, \hE)|^2 \le C \int_{\mR^2} |(a(\xi), \ha(\xi)|^2 \sqrt{|\xi|^{2} + |k|^2}. 
\end{equation}
Combining  \eqref{pro-HS-est2}, \eqref{pro-HS-est3}, and \eqref{pro-HS-est4} yields 
\begin{multline}\label{pro-HS-est5}
\int_{\mR^3_+} |(\nabla E, \nabla \hE)|^2  + |k|^2 |(E, \hE)|^2  \\[6pt]
\le C \left( \| (f_e, f_m) \|_{H^{1/2}(\mR^3_0)}^2 + |k| \,  \| (f_e, f_m) \|_{L^2(\mR^3_0)}^2 + \frac{1}{|k|^2}  \| (\dive_\Gamma f_e, \dive_\Gamma f_m) \|_{H^{1/2}(\mR^3_0)}^2 \right). 
\end{multline}

Similarly, we obtain  
\begin{multline}\label{pro-HS-est6}
\int_{\mR^3_+} |(\nabla H, \nabla \hH)|^2  + |k|^2 |(H, \hH)|^2  \\[6pt]
\le C \left( \| (f_e, f_m) \|_{H^{1/2}(\mR^3_0)}^2 + |k|  \, \| (f_e, f_m) \|_{L^2(\mR^3_0)}^2 + \frac{1}{|k|^2}  \| (\dive_\Gamma f_e, \dive_\Gamma f_m) \|_{H^{1/2}(\mR^3_0)}^2 \right). 
\end{multline}

The conclusion now follows from \eqref{pro-HS-est5} and \eqref{pro-HS-est6}. The proof is complete. 
\end{proof}

As a consequence of \Cref{pro-HS}, we obtain 

\begin{corollary}\label{cor-HS}  

 Let $\alpha \in \mC$ with $\alpha^2 \in \mR$ and $|\alpha| = 1$,  $\gamma > 0$,  $k  \in \mC$ with $|\Im(k^2)| \ge \gamma |k|^2$, $|\Im(\alpha^2 k^2)| \ge \gamma |k|^2$,  and $|k| \ge 1$, and   let $\eps, \,  \mu,  \, \heps, \,  \hmu \in [L^\infty(\mR^3_+)]^{3 \times 3}$ be symmetric, uniformly elliptic,  and of class $C^1$. Let $\Lambda \ge 1$ be such that 
$$
\Lambda^{-1} \le \eps, \, \mu, \,  \heps, \,  \hmu  \le \Lambda \mbox{ in } B_1 \cap \mR^3_+  \quad  \mbox{ and } \quad \|(\eps, \mu, \heps, \hmu)  \|_{C^1(\mR^3_+ \cap B_1)} \le \Lambda. 
$$
Assume that $\eps(0), \, \heps(0), \, \mu(0), \hmu(0)$ are isotropic,  and for some $\Lambda_1 \ge 0$
$$
|\eps (0) -  \heps(0)| \ge \Lambda_1, \quad |\mu(0) - \hmu(0)| \ge \Lambda_1, \quad \mbox{ and } \quad |\eps(0)/ \mu (0) - \heps(0)/ \hmu(0)| \ge \Lambda_1. 
$$
Let $J_e, J_m, \hJ_e, \hJ_m \in H(\dive, \mR^3_+)$ and assume that  $(E, H, \hE, \hH) \in [L^2(\mR^3)]^{12}$ is a solution  of \eqref{sys-C*} and \eqref{bdry-C*} with $f_e = f_m = 0$. There exist $0 < r_0 < 1$ and $k_0 > 1$  depending only on $\gamma$, $\Lambda$, and $\Lambda_1$ such that if   the supports of $E, \, H, \, \hE, \, \hH$ are in $B_{r_0} \cap \overline{\mR^3_+}$ and $|k| \ge k_0$, then 
\begin{multline}\label{cor-HS-est}
\| (E, H, \hE, \hH) \|_{H^1(\mR^3_+)} + |k| \,  \| (E, H, \hE, \hH) \|_{L^2(\mR^3_+)} 
 \le C \Big(   \| (J_e, J_m, \hJ_e, \hJ_m)\|_{L^2(\mR^3_+)} \\[6pt] + \frac{1}{|k|} \| (\dive J_e, \dive J_m, \dive \hJ_e, \dive \hJ_m)\|_{L^2(\mR^3_+)} +  \frac{1}{|k|} \| (J_{e, 3} - \hJ_{e, 3}, J_{m, 3} -  \hJ_{m, 3})\|_{H^{1/2}(\mR^3_0)}  \Big), 
\end{multline}
for some positive constant $C$ depending only on $\gamma$,  $\Lambda$,  and $\Lambda_1$. 
\end{corollary}

\begin{remark}  \rm The constant $C$ in \eqref{cor-HS-est} is independent of $k$. Concerning the isotropic properties of $\eps, \, \mu, \, \heps, \, \hmu$, 
we emphasize here that $\eps, \, \mu, \, \heps, \, \hmu$ are not required to be isotropic in $B_1 \cap \overline{\mR^3_+}$, we only assume that $\eps(0), \, \mu(0), \, \heps(0), \, \hmu(0)$ are.  
\end{remark}
Here and what follows $B_r$, for $r>0$ denotes the ball of radius $r$ centered at the origin.

\begin{proof}  We rewrite \eqref{sys-C*} under the form
\begin{equation*}
\left\{\begin{array}{c}
\nabla \times E = k \mu (0) H + J_e^{1}
\mbox{ in } \mR^3_+, \\[6pt]
\nabla \times H = - k \eps (0) E + J_m^{1} \mbox{ in } \mR^3_+, 
\end{array}\right. \quad \left\{\begin{array}{c}
\nabla \times \hE = \alpha k \hmu(0) \hH + \hJ_e^{1}   \mbox{ in } \mR^3_+, \\[6pt]
\nabla \times \hH = - \alpha k \heps (0) \hE + \hJ_m^1 \mbox{ in } \mR^3_+, 
\end{array}\right.
\end{equation*}
where, in $\mR^3_+$, 
$$
J_e^{1} (x)  = J_e (x) + k (\mu(x) - \mu(0) ) H (x) , \quad J_m^{1} (x)  = J_m (x)  - k (\eps(x) - \eps(0) ) E (x) , 
$$ 
$$
\hJ_e^{1} (x)  = \hJ_e (x)  + \alpha k (\hmu(x) - \hmu(0) ) \hH (x) , \quad \hJ_m^{1} (x)  = \hJ_m (x)  - \alpha k (\heps(x) - \heps(0) ) \hE (x) . 
$$  
From \Cref{pro-HS}, we obtain 
\begin{multline}\label{cor-HS-p1}
C \Big( \| (E, H, \hE, \hH) \|_{H^1(\mR^3_+)} + |k| \, \| (E, H, \hE, \hH) \|_{L^2(\mR^3_+)}  \Big) \\[6pt]
\le  \| (J_e^1, J_m^1, \hJ_e^1, \hJ_m^1)\|_{L^2(\mR^3_+)}  
+ \frac{1}{|k|}  \| (\dive J_e^1, \dive J_m^1, \dive \hJ_e^1, \dive \hJ_m^1)\|_{L^2(\mR^3_+)} \\[6pt]
 + \frac{1}{|k|}  \| (J_{e, 3}^1 - \hJ_{e, 3}^1, J_{m, 3}^1 - \hJ_{m, 3}^1) \|_{H^{1/2}(\mR^3_0)}. 
\end{multline}

On the other hand, from the definition of $(J_e^1, J_m^1, \hJ_e^1, \hJ_m^1)$, one has  
\begin{equation}\label{cor-HS-p2}
\| (J_e^1, J_m^1, \hJ_e^1, \hJ_m^1)\|_{L^2(\mR^3_+)} 
\le C  \| (J_e, J_m, \hJ_e, \hJ_m)\|_{L^2(\mR^3_+)}  + C  r_0 |k| \, \| (E, H, \hE, \hH) \|_{L^2(\mR^3_+)}
\end{equation}
\begin{multline}\label{cor-HS-p3}
\frac{1}{|k|}  \| (\dive J_e^1, \dive J_m^1, \dive \hJ_e^1, \dive \hJ_m^1)\|_{L^2(\mR^3_+)} 
 \le  \frac{C}{|k|}  \| \big(\dive J_e, \dive J_m, \dive \hJ_e, \dive \hJ_m \big) \|_{L^2(\mR^3_+)}  \\[6pt]
 +  C  \| (E, H, \hE, \hH) \|_{L^2(\mR^3_+)} 
 +  C r_0 \left\|  \left(\nabla E, \nabla H, \nabla \hE, \nabla \hH \right) \right\|_{L^2(\mR^3_+)}, 
\end{multline}
and 
\begin{multline}\label{cor-HS-p4}
 \frac{1}{|k|}  \| (J_{e, 3}^1 - \hJ_{e, 3}^1, J_{m, 3}^1 - \hJ_{m, 3}^1) \|_{H^{1/2}(\mR^3_0)} \le  \frac{C}{|k|}  \| (J_{e, 3} - \hJ_{e, 3}, J_{m, 3} - \hJ_{m, 3}) \|_{H^{1/2}(\mR^3_0)} \\[6pt]
 + C r_0 \| (E, H, \hE, \hH) \|_{H^{1}(\mR^3_+)} + C  \| (E, H, \hE, \hH) \|_{L^2(\mR^3_+)}. 
\end{multline}
Here in the last inequality, we involved the trace theory and  used  
\begin{multline*}
 \left\| \left((\mu (x) - \mu(0)) H , (\eps (x) - \eps(0)) E, (\hmu (x) - \hmu(0)) \hH,  (\heps (x) - \heps(0)) \hE  \right) \right\|_{H^{1}(\mR^3_+)}  \\[6pt]
 \le C r_0 \| (E, H, \hE, \hH) \|_{H^{1}(\mR^3_+)} + C \| (E, H, \hE, \hH) \|_{L^2(\mR^3_+)}. 
\end{multline*}
 
Combining \eqref{cor-HS-p1} - \eqref{cor-HS-p4} yields 
\begin{multline}\label{cor-HS-p5}
C \Big( \| (E, H, \hE, \hH) \|_{H^1(\mR^3_+)} + |k| \, \| (E, H, \hE, \hH) \|_{L^2(\mR^3_+)}  \Big) \\[6pt]
\le  \| (J_e, J_m, \hJ_e, \hJ_m)\|_{L^2(\mR^3_+)}  
+ \frac{1}{|k|}  \| (\dive J_e, \dive J_m, \dive \hJ_e, \dive \hJ_m)\|_{L^2(\mR^3_+)}   \\[6pt] +  \frac{1}{|k|} \| (J_{e, 3} -  \hJ_{e, 3}, J_{m, 3} -  \hJ_{m, 3})\|_{H^{1/2}(\mR^3_0)}  \\[6pt]
 (|k| r_0 +  1) \| (E, H, \hE, \hH) \|_{L^2(\mR^3_+)} 
 +   r_0 \|  ( E, H, \hE, \hH ) \|_{H^1(\mR^3_+)}. 
\end{multline}
Fix $r_0 = \min\{C / 4, 1/4\}$ where $C$ is the constant in \eqref{cor-HS-p5}. Take $k_0$ such that $k_0 r_0  \ge 1$. One then can absorb the last two terms of the RHS of \eqref{cor-HS-p5} by the LHS. The conclusion then  follows. 
\end{proof}

\section{Proof of \Cref{thm-main}} \label{sect-proof}

In this section, we give the proof of \Cref{thm-main}. We begin with a result which yields  the uniqueness  and the stability of \eqref{sys-ITE-E} and \eqref{bdry-ITE-E}.  

\begin{proposition}\label{pro-main-1} Let $\gamma > 0$,  $\alpha \in \mC$ with $\alpha^2 \in \mR$ and $|\alpha|=1$, and let $\eps, \,  \mu, \,  \heps, \,  \hmu \in [L^\infty(\Omega)]^{3 \times 3}$ be symmetric. Assume that there exist $\Lambda  \ge 1$, $\Lambda_1 > 0$, and $s_0 > 0$   such that 
\begin{equation*}
\Lambda^{-1} \le \eps, \, \mu, \, \heps, \,  \hmu \le \Lambda \mbox{ a.e. in } \Omega,  \quad \| (\eps, \mu, \heps, \hmu)  \|_{C^1(\bar \Omega_{s_0})} \le \Lambda, 
\end{equation*}
$$
\eps,\,  \mu,\,  \heps, \,  \hmu \mbox{ are isotropic on } \partial \Omega, 
$$
and, for $x \in \partial \Omega$,  
$$
|\eps (x) -  \heps(x)| \ge \Lambda_1, \quad |\mu(x) - \hmu(x)| \ge \Lambda_1, \quad \mbox{ and } \quad |\eps(x)/ \mu (x) - \heps(x)/ \hmu(x)| \ge \Lambda_1. 
$$
There exist two positive constants $k_0 \ge 1$ and $C > 0$ depending only on $\Lambda$, $\Lambda_1$, $s_0$, $\gamma$,  and $\Omega$ such that for $k \in \mC$ with $|\Im{(k^2)}| \ge \gamma |k|^2$ and $|k| \ge k_0$,  for every $(J_e, J_m, \hJ_e, \hJ_m) \in [H(\dive, \Omega)]^4$ with $(J_{e}\cdot \nu - \hJ_{e} \cdot \nu, J_{m} \cdot \nu -  \hJ_{m} \cdot \nu) \in [H^{1/2}(\partial \Omega)]^2$,  and for every solution  $(E, H, \hE, \hH) \in [L^2(\Omega)]^{12}$ of 
\begin{equation}\label{sys-ITE-1}
\left\{\begin{array}{c}
\nabla \times E = k \mu H + J_e \mbox{ in } \Omega, \\[6pt]
\nabla \times H = - k \eps E + J_m \mbox{ in } \Omega, 
\end{array}\right. \quad \left\{\begin{array}{c}
\nabla \times \hE = \alpha k \hmu \hH + \hJ_e \mbox{ in } \Omega, \\[6pt]
\nabla \times \hH = - \alpha k \heps \hE + \hJ_m \mbox{ in } \Omega,  
\end{array}\right.
\end{equation}
\begin{equation}\label{bdry-ITE-1}
(\hE - E) \times \nu = 0 \mbox{ on } \partial \Omega, \quad \mbox{ and } \quad (\hH - H) \times \nu = 0 \mbox{ on } \partial \Omega, 
\end{equation}
we have 
\begin{multline}\label{pro-main-1-est}
 |k| \,  \| (E, H, \hE, \hH) \|_{L^2(\Omega)}  + \| (E, H, \hE, \hH) \|_{H^1(\Omega_{s_0/2})} 
 \le C \| (J_e, J_m, \hJ_e, \hJ_m)\|_{L^2(\Omega)} \\[6pt]
 + \frac{C}{|k|}  \| (\dive J_e,\dive  J_m, \dive \hJ_e,\dive  \hJ_m)\|_{L^2(\Omega)} +  \frac{C}{|k|}  \| (J_{e} \cdot \nu - \hJ_e \cdot \nu, J_m \cdot \nu -  \hJ_m \cdot \nu)\|_{H^{1/2}(\Omega)}  . 
\end{multline}
\end{proposition}

Recall that $\Omega_s$ is given in \eqref{def-Os}. 

\begin{proof} We use local charts for $\Gamma  = \partial \Omega$. 
In what follows, we denote  $Q = (-1, 1)^3$, $Q_{+} = Q \cap \mR^3_+$, and  $Q_{0}= Q  \cap \mR^3_0$.   

Let $m \ge 1$ and let   $\varphi_\ell \in  C^2_{c}(\mR^3)$, $U_\ell \subset \mR^3$  open ball, and $\cT_\ell : U_\ell \to Q$ with $1 \le \ell \le m$ be such that 
$ \cT_\ell (U_\ell \cap \Omega) = Q_+$, and $\cT_\ell(U_\ell \cap \Gamma) = Q_0$, $\supp \varphi_\ell \Subset U_\ell$, and $\Phi  = 1$ in a neighborhood of $\Gamma$, where 
$$
\Phi : = \sum_{\ell=1}^m \varphi_\ell \mbox{ in } \mR^3.
$$
In what follows, we also assume that the diameter of the support of $\varphi_\ell$ is sufficiently small and $\nabla \cT_\ell (\varphi_\ell^{-1}(0))$ is a rotation, i.e., $\big( \nabla \cT_\ell \nabla \cT_\ell^T \big)  (\varphi_\ell^{-1}(0))= I$. 
Set, in  $\Omega \cap U_\ell$,  
$$
(E^\ell, H^\ell, \hE^\ell, \hH^\ell) = (\varphi_\ell E, \varphi_\ell H, \varphi_\ell \hE, \varphi_\ell  \hH), 
$$
and 
\begin{equation*}
(J_e^\ell , J_m^\ell , \hJ_e^\ell , \hJ_m^\ell) = (\varphi_\ell  J_e  + \nabla \varphi_\ell  \times E, \varphi_\ell  J_m  +  \nabla \varphi_\ell  \times H, \varphi_\ell  \hJ_e  + \nabla \varphi_\ell  \times \hE, \varphi_\ell  \hJ_m  +  \nabla \varphi_\ell  \times \hH). 
\end{equation*}
We have 
\begin{equation}
\left\{\begin{array}{c}
\nabla \times E^\ell = k \mu H^\ell + J_e^\ell \mbox{ in } \Omega \cap U_\ell, \\[6pt]
\nabla \times H^\ell = - k \eps E^\ell + J_m^\ell \mbox{ in } \Omega \cap U_\ell, 
\end{array}\right. \quad \left\{\begin{array}{c}
\nabla \times \hE^\ell = \alpha k \hmu \hH^\ell + \hJ_e^\ell \mbox{ in } \Omega \cap U_\ell, \\[6pt]
\nabla \times \hH^\ell = - \alpha k \heps \hE^\ell + \hJ_m^\ell \mbox{ in } \Omega \cap U_\ell,  
\end{array}\right.
\end{equation}
\begin{equation}
(\hE^\ell - E^\ell) \times \nu = 0 \mbox{ on } \partial \Omega \cap U_\ell, \quad \mbox{ and } \quad (\hH^\ell - H^\ell) \times \nu = 0 \mbox{ on } \partial \Omega \cap U_\ell.  
\end{equation}

Given a diffeomorphism $\cT$ from an open $D$ onto an open $D'$, the following standard notations are used 
$$
\cT* u(x') = \nabla \cT(x) u(x), 
$$
\begin{equation*}
\cT_*a (x') = \frac{\nabla \cT(x) a(x) \nabla \cT^T(x)}{\det \nabla \cT(x)},  \quad \mbox{ and } \quad  \cT_*j (x')= \frac{\nabla \cT(x) j(x)}{\det \nabla \cT (x)},
\end{equation*}
with  $x' =\cT(x)$,  for a matrix-valued function $a$,  and  for vector fields $u$ and  $j$ defined in $D$.  
Set, in $Q_+$,  
$$
(\tE^\ell, \tH^\ell, \thE^\ell, \thH^\ell) =   \big(\cT_\ell * E^\ell, \cT_\ell * H^\ell, \cT_\ell * \hE^\ell, \cT_\ell * \hH^\ell \big), 
$$
$$
(\eps^\ell, \mu^\ell, \heps^\ell, \hmu^\ell) = ({\cT_\ell}_* \eps, {\cT_\ell}_* \mu, {\cT_\ell}_* \heps, {\cT_\ell}_* \hmu), 
$$
$$
(\tJ_e^\ell, \tJ_m^\ell, \thJ_e^\ell, \thJ_m^\ell) = ({\cT_\ell}_* J_e^\ell, {\cT_\ell}_* J_m^\ell, {\cT_\ell}_* \hJ_e^\ell, {\cT_\ell}_* \hJ_m^\ell). 
$$

By a change of variables, see e.g. \cite[Lemma 7]{Ng-Superlensing-Maxwell}, 
\begin{equation}\label{sys-ITE-2}
\left\{\begin{array}{c}
\nabla \times \tE^\ell = k \mu^\ell \tH^\ell + \tJ_e^\ell \mbox{ in }  Q_+, \\[6pt]
\nabla \times \tH^\ell = - k \eps^\ell \tE^\ell + \tJ_m^\ell \mbox{ in } Q_+, 
\end{array}\right. \quad \left\{\begin{array}{c}
\nabla \times \thE^\ell = \alpha k \hmu \thH^\ell + \thJ_e^\ell \mbox{ in } Q_+, \\[6pt]
\nabla \times \thH^\ell = - \alpha k \heps \thE^\ell + \thJ_m^\ell \mbox{ in } Q_+,  
\end{array}\right.
\end{equation}
\begin{equation}\label{bdry-ITE-2}
(\thE^\ell - \tE^\ell) \times \nu = 0 \mbox{ on } Q_0, \quad \mbox{ and } \quad (\thH^\ell - \tH^\ell) \times \nu = 0 \mbox{ on } Q_0.  
\end{equation}

Since $\nabla \cT_\ell (\varphi_\ell^{-1}(0))$ is a rotation,  and 
$\eps, \, \mu, \, \heps, \, \hmu$ are isotropic on $\partial \Omega$, one has 
$$
\eps^\ell (0), \,  \mu^\ell (0), \,  \heps^\ell(0), \,  \hmu^\ell (0) \mbox{ are isotropic}. 
$$
By considering the diameter of $\supp \varphi_\ell$ sufficiently small, one can then apply \Cref{cor-HS} to $(\tE^\ell, \tH^\ell, \thE^\ell, \thH^\ell)$.  We then obtain 
\begin{multline}\label{pro-main-e1}
C \Big( \| ( \tE^\ell, \tH^\ell, \thE^\ell, \thH^\ell) \|_{H^1(Q_+)}  + |k| \,  \| ( \tE^\ell, \tH^\ell, \thE^\ell, \thH^\ell) \|_{L^2(Q_+)} \Big) \\[6pt]  \le  \| (\tJ_e^\ell, \tJ_m^\ell, \thJ_e^\ell, \thJ_m^\ell)\|_{L^2(Q_+)} + \frac{1}{|k|}  \| (\dive  \tJ_e^\ell, \dive  \tJ_m^\ell, \dive  \thJ_e^\ell, \dive  \thJ_m^\ell)\|_{L^2(Q_+)} \\[6pt]
+ \frac{1}{|k|}  \| (\tJ_e^\ell \cdot e_3 -  \thJ_e^\ell \cdot e_3, \tJ_m^\ell \cdot e_3 -  \thJ_m^\ell \cdot e_3)\|_{H^{1/2}(Q_0)}. 
\end{multline}

We have, by \cite[Corollary 3.59]{Monk03}, 
$$
\| (\dive \tJ_e^\ell, \dive \tJ_m^\ell, \dive \thJ_e^\ell, \dive \thJ_m^\ell)\|_{L^2(Q_+)} \le C \| (\dive J_e^\ell, \dive J_m^\ell, \dive \hJ_e^\ell, \dive \hJ_m^\ell)\|_{L^2(\Omega \cap U_\ell)}
$$
and we also obtain 
$$
\| (\tJ_e^\ell \cdot e_3 -  \thJ_e^\ell \cdot e_3, \tJ_m^\ell \cdot e_3 -  \thJ_m^\ell \cdot e_3)\|_{H^{1/2}(Q_0)} \le C \|(J_e^\ell \cdot \nu - \hJ_e^\ell \cdot \nu, J_m^\ell \cdot \nu - \hJ_m^\ell \cdot \nu \|_{H^{1/2}(\partial \Omega \cap U_\ell)}. 
$$
We deduce from \eqref{pro-main-e1} that 
\begin{multline}\label{pro-main-e2}
C \Big( \| ( E^\ell, H^\ell, \hE^\ell, \hH^\ell) \|_{H^1(\Omega \cap U_\ell)}  + |k| \,  \| ( E^\ell, H^\ell, \hE^\ell, \hH^\ell) \|_{L^2(\Omega \cap U_\ell)} \Big) \\[6pt]  \le  \| (J_e^\ell, J_m^\ell, \hJ_e^\ell, \hJ_m^\ell)\|_{L^2(\Omega \cap U_\ell)} + \frac{1}{|k|}  \| (\dive  J_e^\ell, \dive  J_m^\ell, \dive  \hJ_e^\ell, \dive  \hJ_m^\ell)\|_{L^2(\Omega \cap U_\ell)} \\[6pt]
+ \frac{1}{|k|}  \| (J_e^\ell \cdot \nu -  \hJ_e^\ell  \cdot \nu , J_m \cdot \nu -  \hJ_m \cdot \nu)\|_{H^{1/2}(\partial \Omega \cap U_\ell)}. 
\end{multline}
Take the sum with respect to $\ell$.  We then  have, for some $\tau_0 <   s_0/4$,  
\begin{multline}\label{pro-main-p1}
C \Big( \| ( E, H, \hE, \hH) \|_{H^1(\Omega_{\tau_0})}  + |k| \, \| (E, H, \hE, \hH) \|_{L^2(\Omega_{\tau_0})}  \Big) \\[6pt]  \le  \| (J_e, J_m, \hJ_e, \hJ_m)\|_{L^2(\Omega)} + \frac{1}{|k|}  \| (\dive  J_e, \dive  J_m, \dive  \hJ_e, \dive  \hJ_m)\|_{L^2(\Omega)} \\[6pt]+ 
\frac{1}{|k|}  \| (J_e \cdot \nu -  \hJ_e \cdot \nu, J_m \cdot \nu -  \hJ_m \cdot \nu)\|_{H^{1/2}(\partial \Omega)}  \\[6pt]
+  \| (E, H, \hE, \hH) \|_{L^2(\Omega_{s_0/2})} + \frac{1}{|k|}  \| ( E,  H,  \hE,  \hH) \|_{H^1(\Omega_{s_0/2})}. 
\end{multline}

Applying  \Cref{lem-decay} below, we have 
\begin{equation}\label{pro-main-p2}
\| (E, H, \hE, \hH) \|_{L^2(\Omega \setminus \Omega_{\tau_0})} \le c_1 e^{- c_2|k|} \| (E, H, \hE, \hH) \|_{L^2(\Omega_{\tau_0})}, 
\end{equation}
for some positive constants $c_1, c_2$ depending only on $\Lambda$, $\gamma$, $\tau_0$,  and $\Omega$. 
Since $(\eps, \mu, \heps, \hmu) \in C^1(\Omega_{3\tau_0} \setminus \Omega_{\tau_0/2})$, it follows from \eqref{sys-ITE-1} that 
\begin{multline}\label{pro-main-p3}
\| (E, H,  \hE, \hH) \|_{H^1(\Omega_{s_0/2} \setminus \Omega_{\tau_0})} \le C |k|  \,  \| (E, H, \hE, \hH) \|_{L^2(\Omega_{s_0} \setminus \Omega_{\tau_0/2})} \\[6pt]
+ C\| (J_e, J_m, \hJ_e, \hJ_m)\|_{L^2(\Omega)} + \frac{C}{|k|}  \| (\dive  J_e, \dive  J_m, \dive  \hJ_e, \dive  \hJ_m)\|_{L^2(\Omega)}. 
\end{multline}
Taking $k_0$ sufficiently large and $|k| \ge k_0$, from \eqref{pro-main-p2} and \eqref{pro-main-p3}, one can absorb the last two terms of the RHS of \eqref{pro-main-p1} by the LHS of \eqref{pro-main-p1}. We then have 
\begin{multline}\label{pro-main-p4}
C \Big( \| (E, H, \hE, \hH) \|_{H^1(\Omega_{\tau_0})}  + |k| \, \| (E, H, \hE, \hH) \|_{L^2(\Omega_{\tau_0})} \Big) \\[6pt] 
 \le  \| (J_e, J_m, \hJ_e, \hJ_m)\|_{L^2(\Omega)} + \frac{1}{|k|}  \| (\dive  J_e, \dive  J_m, \dive  \hJ_e, \dive  \hJ_m)\|_{L^2(\Omega)} \\[6pt]
 +\frac{1}{|k|}  \| (J_e \cdot \nu -  \hJ_e \cdot \nu, J_m \cdot \nu -  \hJ_m \cdot \nu)\|_{H^{1/2}(\partial \Omega)}.     
\end{multline}
Using \eqref{pro-main-p2} and \eqref{pro-main-p4}, we derive from \eqref{pro-main-p3} that 
\begin{multline}\label{pro-main-p5}
C \| (E, H,  \hE, \hH) \|_{H^1(\Omega_{s_0/2} \setminus \Omega_{\tau_0})}  
 \le  \| (J_e, J_m, \hJ_e, \hJ_m)\|_{L^2(\Omega)} \\[6pt]
 + \frac{1}{|k|}  \| (\dive  J_e, \dive  J_m, \dive  \hJ_e, \dive  \hJ_m)\|_{L^2(\Omega)} 
 +\frac{1}{|k|}  \| (J_e \cdot \nu -  \hJ_e \cdot \nu, J_m \cdot \nu -  \hJ_m \cdot \nu)\|_{H^{1/2}(\partial \Omega)}.     
\end{multline}
The conclusion now follows from \eqref{pro-main-p4} and \eqref{pro-main-p5}. The proof  is complete. 
\end{proof}

In the proof of \Cref{pro-main-1}, we used the following decay result on the Maxwell equations: 

\begin{lemma}\label{lem-decay}  Let $\gamma > 0$, $k  \in \mC$ with $|\Im(k^2)| \ge \gamma |k|^2$ and $|k| \ge 1$, and let $\eps, \mu \in [L^\infty(\Omega)]^{3\times 3}$ be symmetric and uniformly elliptic, i.e. 
$$
\Lambda^{-1} \le \eps, \, \mu, \, \heps, \hmu \le \Lambda, 
$$
for some $\Lambda \ge 1$.  Given $J_e, J_m \in L^2(\Omega)$,  let  $(E, H) \in [L^2(\Omega)]^{6}$ be a solution  of 
\begin{equation*}
\left\{\begin{array}{c}
\nabla \times E= k \mu H + J_e \mbox{ in } \Omega, \\[6pt]
\nabla \times H = - k \eps E + J_m \mbox{ in } \Omega. 
\end{array}\right. 
\end{equation*}
For all $s>0$, there exists two positive constants $c_1$ and $c_2$ depending only on $\Lambda$, $\gamma$, $s$, and $\Omega$ such that 
\begin{equation*}
\| (E, H)\|_{L^2(\Omega \setminus \Omega_s)} \le c_1\exp (-c_2 |k|  ) \| (E, H)\|_{L^2(\Omega_s)} + c_1 \| (J_e, J_m)\|_{L^2(\Omega)}. 
\end{equation*}
\end{lemma}

\begin{proof} Let $(E^1, H^1) \in [L^2(\Omega)]^6$ be the unique solution of 
\begin{equation}\label{lem-HS-sys}
\left\{\begin{array}{c}
\nabla \times E^1= k \mu H^1 + J_e \mbox{ in } \Omega, \\[6pt]
\nabla \times H^1 = - k \eps E^1 + J_m \mbox{ in } \Omega, \\[6pt]
E^1 \times e_3 = 0  \mbox{ on } \partial \Omega,  
\end{array}\right. 
\end{equation}
As in the proof of  \Cref{lem-HS}, we have 
$$
\| (E^1, H^1) \|_{L^2(\Omega)} \le C \| (J_e, J_m)\|_{L^2(\Omega)}. 
$$
Considering $(E- E^1, H-H^1)$, w.l.o.g., one might assume that $J_e = J_m  = 0$ in $\Omega$. This is assumed from later on.  

Fix $\varphi \in C^2(\Omega)$  such that $\varphi = c s$ in $\Omega \setminus \Omega_s$ and $\varphi = 0$ in $\Omega_{s/2}$, and $|\nabla \varphi| \le c$  in $\Omega$ where  $c$ is a small positive constant defined later (the smallness of $c$ depends only on $\Lambda$ and $\Omega$, it is independent of $s$). 
Set $\phi(x) =  e^{ |k| \varphi(x)}$ and   $E^1 (x)= \phi(x)E(x)$ and $H^1(x) = \phi (x) H(x)$ for $x \in \Omega$. 
We have 
\begin{equation*}
\left\{\begin{array}{c}
\nabla \times E^1= k \mu H^1 + J_e^1 \mbox{ in } \Omega, \\[6pt]
\nabla \times H^1 = - k \eps E^1 + J_m^1  \mbox{ in } \Omega, 
\end{array}\right. 
\end{equation*}
where 
$$
J_e^1 =  \nabla \phi \times E \quad \mbox{ and } \quad J_m^1 = \nabla \phi \times H \mbox{ in } \Omega. 
$$
Multiplying the first equation with $\bar H^1$, integrating by parts in $\Omega \setminus \Omega_\tau$ for $s/4 < \tau < s/2$, and using the second equation,  we have 
\begin{multline*}
\left| \int_{\Omega \setminus \Omega_\tau} k \langle \mu H^1, H^1 \rangle   +  \bar k   \int_{\Omega \setminus \Omega_\tau} \langle \eps E^1, E^1 \rangle  \right|\\[6pt]
\le  \int_{\Omega \setminus \Omega_\tau} |J_e^1| |H^1| + |J_m^1| |E^1|   + C \int_{\partial (\Omega \setminus \Omega_\tau)} (|E^1|^2 + |H^1|^2). 
\end{multline*}
This yields 
\begin{multline}\label{lem-decay-bilinear2}
\left| \int_{\Omega \setminus \Omega_\tau} k^2 \langle \mu H^1, H^1 \rangle   +   |k|^2   \int_{\Omega \setminus \Omega_\tau} \langle \eps E^1, E^1 \rangle  \right|\\[6pt]
\le |k| \int_{\Omega \setminus \Omega_\tau} |J_e^1| |H^1| + |J_m^1| |E^1|   + C |k| \int_{\partial (\Omega \setminus \Omega_\tau)} (|E^1|^2 + |H^1|^2). 
\end{multline}
By the definition of $J_e^1, \, J_m^1$, and of $E^1$ and $H^1$,  
$$
|J_e^1| \le c |k| |E^1|, \quad  |J_m^1| \le c |k| |H^1|  \mbox{ in } \Omega, \quad \mbox{ and } \quad E^1 - E = H^1 - H = 0 \mbox{ in } \Omega_{s/2}, 
$$
we derive from \eqref{lem-decay-bilinear2} that, for $c$ sufficiently small, 
\begin{equation*}
\left| \int_{\Omega \setminus \Omega_\tau} k^2 \langle \mu H^1, H^1 \rangle   +  |k|^2  \int_{\Omega \setminus \Omega_\tau} \langle \eps E^1, E^1 \rangle  \right| \le     C |k| \int_{\partial (\Omega \setminus \Omega_\tau)} (|E|^2 + |H|^2). 
\end{equation*}
The conclusion follows by taking $\tau$ such that 
\begin{equation*}
\int_{\partial (\Omega \setminus \Omega_\tau)} (|E|^2 + |H|^2) \le C s^{-1} \int_{\Omega_s} (|E|^2 + |H|^2). 
\end{equation*}
This yields 
\begin{equation*}
\int_{\partial (\Omega \setminus \Omega_\tau)} (|E^1|^2 + |H^1|^2) \le  C s^{-1} \int_{\Omega_s} (|E|^2 + |H|^2), 
\end{equation*}
and the conclusion follows by the definition of $E^1$ and $H^1$. The proof is complete. 
\end{proof}

\begin{remark} \rm The proof of \Cref{lem-decay} is quite standard, see e.g. \cite[Theorem 2.2]{HJNg1} for a variant dealing with  the Helmholtz equation. 
\end{remark}

We now apply a limiting absorption principle. To this end, let us set 
$$
\bH_1 (\Omega) = \Big\{ (u, v) \in [H(\curl, \Omega)]^2; (u -v) \times \nu = 0 \mbox{ on } \partial \Omega \Big\}. 
$$
One can check that $\bH_1(\Omega)$ is a Hilbert space equipped with the natural scalar product induced from the one of $[H(\curl, \Omega)]^2$.

Consider $(J_e, J_m, \hJ_e, \hJ_m) \in [L^2(\Omega)]^{12}$.  Let $\gamma > 0$, $k \in \mC$ with $|\Im(k^2)| \ge \gamma |k|^2$ and $|k| \ge 1$. Set $a = \mbox{ sign }  \Im(k^2)$.  For $\delta > 0$ sufficiently small with respect to $\gamma$, we claim that there exists a unique solution  $(E^\delta, H^\delta, \hE^\delta, \hH^\delta) \in [L^2(\Omega)]^{12}$ of the system 
\begin{equation}\label{sys-ITE-d}
\left\{\begin{array}{c}
\nabla \times E^\delta = (1 - i a \delta ) k \mu H ^\delta+ J_e \mbox{ in } \Omega, \\[6pt]
\nabla \times H^\delta = - (1 - i a \delta ) k \eps E^\delta + J_m \mbox{ in } \Omega, 
\end{array}\right. \quad \left\{\begin{array}{c}
\nabla \times \hE = - (1 - i a\delta) k \hmu \hH^\delta + \hJ_e \mbox{ in } \Omega, \\[6pt]
\nabla \times \hH^\delta =   (1 - i a \delta)  k \heps \hE^\delta + \hJ_m \mbox{ in } \Omega,  
\end{array}\right.
\end{equation}
\begin{equation}\label{bdry-ITE-d}
(\hE^\delta - E^\delta) \times \nu = 0 \mbox{ on } \partial \Omega, \quad \mbox{ and } \quad (\hH^\delta - H^\delta) \times \nu = 0 \mbox{ on } \partial \Omega.  
\end{equation}
Indeed, consider the following equation
\begin{multline}\label{Step2-bilinear}
\int_{\Omega} \langle (1 - i a \delta)^{-1}\mu^{-1} \nabla \times E^\delta, \nabla \times \varphi \rangle +  k^2 (1 - i a \delta) \langle \eps E^\delta, \varphi \rangle \\[6pt] +  \int_{\Omega}  \langle  (1 -  i a \delta)^{-1}\hmu^{-1} \nabla \times \hE^\delta, \nabla \times \hvarphi \rangle +    k^2 (1 -  i a \delta) \langle \heps \hE^\delta, \hvarphi \rangle \\[6pt] =  \int_{\Omega} \langle (1 - i a\delta)^{-1} \mu^{-1} J_e, \nabla \times \varphi \rangle  + k \langle J_m, \varphi \rangle + \int_{\Omega} \langle (1 - i a\delta)^{-1} \hmu^{-1} \hJ_e, \hvarphi \rangle  -  k \langle \hJ_m, \hvarphi \rangle,  
\end{multline}
for all $(\varphi, \hvarphi) \in \bH_1(\Omega)$. 
Note that, the absolute value of the imaginary part of the LHS of \eqref{Step2-bilinear} with $(\varphi, \hvarphi) = (E^\delta, \hE^\delta) $ is greater than $C \delta \int_{\Omega} |(E^\delta, \nabla \times E^\delta, \hE^\delta, \nabla \times \hE^\delta)|^2$ for $\delta$ sufficiently small.  By Lax-Milgram's theory, there exists a unique solution $(E^\delta, \hE^\delta) \in \bH_1(\Omega)$ for \eqref{Step2-bilinear}. Set, in $\Omega$,  
$$
H^\delta = k^{-1}(1 -  i a \delta)^{-1}  \mu^{-1} \Big( \nabla \times E^\delta - J_e\Big) \quad \mbox{ and } \quad   \hH^\delta = - k^{-1}(1 - i a \delta)  \hmu^{-1} \Big( \nabla \times \hE^\delta - \hJ_e\Big). 
$$
Considering $\varphi, \hvarphi \in C^1_{c}(\Omega)$ in \eqref{Step2-bilinear}, we obtain 
$$
\nabla \times H^\delta = - (1 + i \delta) k \eps E^\delta + J_m \mbox{ in } \Omega \quad \mbox{ and } \quad \nabla \times \hH^\delta = - (1 - i \delta) ik \heps \hE^\delta + \hJ_m \mbox{ in } \Omega. 
$$
This in turn implies that 
$$
(\hH^\delta - H^\delta ) \times \nu = 0 \mbox{ on } \partial \Omega. 
$$
Therefore, the existence and the uniqueness of $(E_\delta, H_\delta)$ are established.  

\medskip 

Next denote
\begin{multline}\label{def-Hspace}
\bH (\Omega) = \Big\{ (u, v, \hu, \hv) \in [L^2(\Omega)]^{12}; \dive (\eps u) = \dive (\mu v) = \dive (\heps \hu) = \dive(\hmu \hv) = 0 \mbox{ in } \Omega,  \\[6pt] \mbox{ and }   
\eps u \cdot \nu  - \heps  \hu \cdot \nu =  \mu v \cdot \nu -   \hmu \hv \cdot \nu  = 0 \mbox{ on } \partial \Omega \Big\}, 
\end{multline}
and let 
$$
\| (u, v, \hu, \hv) \|_{\bH} = \| (u, v, \hu, \hv) \|_{L^2(\Omega)}. 
$$
One can check that $\bH(\Omega)$ is a Hilbert space with the corresponding scalar product. 

\medskip 
We finally  have 

\begin{proposition}\label{pro-main} Let $\gamma > 0$, $\eps, \, \mu, \, \heps, \, \hmu \in [L^\infty(\Omega)]^{3 \times 3}$ be symmetric. Assume that there exist $\Lambda  \ge 1$, $\Lambda_1 > 0$, and $s_0 > 0$   such that 
\begin{equation*}
\Lambda^{-1} \le \eps, \, \mu, \, \heps, \,  \hmu \le \Lambda \mbox{ in } \Omega,  \quad \| (\eps, \mu, \heps, \hmu)  \|_{C^1(\Omega_{s_0})} \le \Lambda, 
\end{equation*}
$$
\eps,\,  \mu,\,  \heps, \,  \hmu \mbox{ are isotropic on } \partial \Omega, 
$$
and, for $x \in \partial \Omega$,  
$$
|\eps (x) -  \heps(x)| \ge \Lambda_1, \quad |\mu(x) - \hmu(x)| \ge \Lambda_1, \quad \mbox{ and } \quad |\eps(x)/ \mu (x) - \heps(x)/ \hmu(x)| \ge \Lambda_1. 
$$
There exist two positive constants $k_0 \ge 1$ and $C > 0$ depending only on $\Lambda$, $\Lambda_1$, $\gamma$, $s_0$, and $\Omega$ such that for $k \in \mC$ with $|\Im{(k^2)}| \ge \gamma |k|^2$ and $|k| \ge k_0$, and for every $(J_e^1, J_m^1, \hJ_e^1, \hJ_m^1) \in \bH(\Omega)$, there exists a unique solution  $(E, H, \hE, \hH) \in \bH(\Omega)$ of 
\begin{equation}\label{sys-ITE-m}
\left\{\begin{array}{c}
\nabla \times E = k \mu H + J_e \mbox{ in } \Omega, \\[6pt]
\nabla \times H = - k \eps E + J_m \mbox{ in } \Omega, 
\end{array}\right. \quad \left\{\begin{array}{c}
\nabla \times \hE = k \hmu \hH + \hJ_e \mbox{ in } \Omega, \\[6pt]
\nabla \times \hH = - k \heps \hE + \hJ_m \mbox{ in } \Omega,  
\end{array}\right.
\end{equation}
\begin{equation}\label{bdry-ITE-m}
(\hE - E) \times \nu = 0 \mbox{ on } \partial \Omega, \quad \mbox{ and } \quad (\hH - H) \times \nu = 0 \mbox{ on } \partial \Omega, 
\end{equation}
where $(J_e, J_m, \hJ_e, \hJ_m) = (\mu J_m^1, \eps J_e^1, \hmu \hJ_m^1, \heps \hJ_e^1)$.  Moreover, 
\begin{equation}\label{pro-main-est}
 |k| \,  \| (E, H, \hE, \hH) \|_{L^2(\Omega)}  + \| (E, H, \hE, \hH) \|_{H^1(\Omega_{s_0/2})} 
 \le C \| (J_e^1, J_m^1, \hJ_e^1, \hJ_m^1)\|_{L^2(\Omega)}. 
\end{equation}
\end{proposition}

\begin{proof} 
For $\delta > 0$, let $(E^\delta, H^\delta, \hE^\delta, \hH^\delta) \in [L^2(\Omega)]^{12}$ be the unique solution of the system \eqref{sys-ITE-d} and \eqref{bdry-ITE-d}. Applying \Cref{pro-main-1}, we have 
\begin{equation}
 |k| \,  \| (E^\delta, H^\delta, \hE^\delta, \hH^\delta) \|_{L^2(\Omega)}  + \| (E^\delta, H^\delta, \hE^\delta, \hH^\delta) \|_{H^1(\Omega_{s_0/2})} 
 \le C \| (J_e, J_m, \hJ_e, \hJ_m)\|_{L^2(\Omega)}. 
\end{equation}
Letting $\delta \to 0$, we obtain the existence of a solution of \eqref{sys-ITE-d} and \eqref{bdry-ITE-d} with $\delta = 0$. Moreover, by \Cref{pro-main-1} again, this solution is unique and it holds, for some $s_0 > 0$,  
\begin{equation}\label{thm-main-p1}
\| (E, H, \hE, \hH) \|_{H^1(\Omega_{s_0})} +  |k| \,  \| (E, H, \hE, \hH) \|_{L^2(\Omega)}  
 \le C \| (J_e, J_m, \hJ_e, \hJ_m)\|_{L^2(\Omega)}. 
\end{equation}

Define the operator 
\begin{eqnarray}
\begin{array}{rcccc}
T_1: &  \bH (\Omega) &\to&  \bH (\Omega) \\[6pt]
& (J_e^1, J_m^1, \hJ_e^1, \hJ_m^1)  &  \mapsto & (E, H, \hE, \hH), 
\end{array}
\end{eqnarray}
It is clear that system \eqref{sys-ITE-m} can be rewritten under the form 
\begin{equation}
\left\{\begin{array}{c}
\nabla \times E = k \mu H + J_e \mbox{ in } \Omega, \\[6pt]
\nabla \times H = - k \eps E + J_m \mbox{ in } \Omega, 
\end{array}\right. \quad \left\{\begin{array}{c}
\nabla \times \hE = - k \hmu \hH + \hJ_e + 2k \hmu \hH \mbox{ in } \Omega, \\[6pt]
\nabla \times \hH = k \heps \hE + \hJ_m - 2 k \heps \hE \mbox{ in } \Omega,  
\end{array}\right.
\end{equation}
Thus, for $(E, H, \hE, \hH) \in \bH(\Omega)$, system \eqref{sys-ITE-m} and \eqref{bdry-ITE-m} is equivalent to 
$$
(E, H, \hE, \hH) = T_1 (J_e^1, J_m^1, \hJ_e^1 - 2 k \hE, \hJ_m^1 + 2k \hH) = T_1 (J_e^1, J_m^1, \hJ_e^1, \hJ_m^1 ) + T_1 (0, 0, - 2 k \hE,  2k \hH). 
$$
Since this equation has at most one  solution by \Cref{pro-main-1} and from the fact that $T_1$ is compact, this equation has a unique solution. The proof is completed by applying again \Cref{pro-main-1} to obtain \eqref{pro-main-est}.    
\end{proof}

We are ready to give the proof of our main theorem.

\begin{proof}[Proof of \Cref{thm-main}]
Define the operator 
\begin{eqnarray}
\begin{array}{rcccc}
T: &  \bH (\Omega) &\to&  \bH (\Omega)\\[6pt]
& (J_e^1, J_m^1, \hJ_e^1, \hJ_m^1)  &  \mapsto & (E, H, \hE, \hH), 
\end{array}
\end{eqnarray}
where $(E, H, \hE, \hH) \in \bH(\Omega)$ is the unique solution of \eqref{sys-ITE-m} and \eqref{bdry-ITE-m} for $k\in {\mathbb C}$ satisfying the assumptions in Proposition \ref{pro-main}.  We claim that $T$ is compact. Indeed, this follows from  
$$
\dive(\eps E) = \dive (\mu H) = \dive (\heps \hE) = \dive(\hmu \hH) = 0. 
$$ 
and \eqref{pro-main-est}.   By  the theory of compact operator see,  e.g., \cite{Brezis-FA}, the spectrum of $T$ is discrete. It is clear that  an eigenfunction pair of the ITE problem corresponding to the eigenvalue $\omega$ is an eigenfunction pair of $T$ corresponding to the eigenvalue $k=i \omega $. Hence, the spectrum of the ITE problem is discrete, and the only possible accumulation point of the transmission eigenvalues is $\infty$ since they coincide with the eigenvalues of the inverse of $T$. 
\end{proof}

Finally, we present 

\begin{proof}[Proof of \Cref{pro}] \Cref{pro} is just a consequence of \Cref{pro-main} by noting that the solution given there is 0 if $(J_e, J_m, \hJ_e, \hJ_m) = 0$. 
\end{proof}

\section*{Acknowledgments}
{The research of F. Cakoni is partially supported  by the AFOSR Grant  FA9550-20-1-0024 and  NSF Grant DMS-1813492. H.-M. Nguyen thanks Fondation des Sciences Math\'ematiques de Paris (FSMP) for the Chaire d'excellence which allows him to visit  Laboratoire Jacques Louis Lions  and Mines ParisTech. This work is completed during this visit.}

\providecommand{\bysame}{\leavevmode\hbox to3em{\hrulefill}\thinspace}
\providecommand{\MR}{\relax\ifhmode\unskip\space\fi MR }
\providecommand{\MRhref}[2]{%
  \href{http://www.ams.org/mathscinet-getitem?mr=#1}{#2}
}
\providecommand{\href}[2]{#2}

\end{document}